\documentclass[11pt]{article}

\usepackage{amssymb}
\usepackage{amsfonts,amsmath,mathtools}
\usepackage{graphics}
\usepackage{lipsum}
\usepackage[ruled]{algorithm2e}
\usepackage{float}
\usepackage{url}
\usepackage{subfig}
\usepackage{xcolor}
\usepackage[export]{adjustbox}
\usepackage{tikz}
\usepackage[margin=1.6in]{geometry}

\newcommand\blfootnote[1]{%
  \begingroup
  \renewcommand\thefootnote{}\footnote{#1}%
  \addtocounter{footnote}{-1}%
  \endgroup
}

\newtheorem{thm}{Theorem}[section]
 
 \newtheorem{lem}[thm]{Lemma}
 \newtheorem{prop}[thm]{Proposition}
 \newtheorem{defn}[thm]{Definition}
\newtheorem{rem}[thm]{Remark}
 \newtheorem{ex}[thm]{Example}
\newtheorem{Probl}[thm]{Problem}
\newtheorem{Charact}[thm]{Characterization}
\newtheorem{Discus}[thm]{\rm{\emph{Discussion}}}
\newtheorem{Const}[thm]{Construction}

\newenvironment{proof}[1][Proof]{\textrm{\em #1.} }{$\Box$\medskip\medskip}

\newcommand\Tor{\operatorname{Tor}}
\newcommand\Char{\operatorname{char}}

\newcommand\Gin{\operatorname{Gin}}
\newcommand\Gap{\operatorname{Gap}}
\newcommand\gap{\operatorname{-gap}}

\newcommand\wdt{{\operatorname{wd}}}

\newcommand\Mon{{\operatorname{Mon}}}

\newcommand\Shad{{\operatorname{Shad}}}

\newcommand\BShad{{\operatorname{BShad}}}
\newcommand\indeg{{\operatorname{indeg}}}

\newcommand\supp{{\operatorname{supp}}}

\newcommand\slex{{\operatorname{slex}}}

\newcommand\m{{\operatorname{max}}}

\def\NZQ{\mathbb}
\def\NN{{\NZQ N}}
\def\ZZ{{\NZQ Z}}
\def\FFF{{\NZQ F}}

\newcommand\bd{\mathbf{ds}}
\newcommand\bl{\mathbf{dl}}
\newcommand\C{{\operatorname{Corn}}}

\begin{document}

\title{On the extremal Betti numbers of squarefree monomial ideals}
\author{Luca Amata, Marilena Crupi*
}  
\newcommand{\Addresses}{{
  \footnotesize
  \textsc{Department of Mathematics and Computer Sciences, Physics and Earth Sciences, University of Messina, Viale Ferdinando Stagno d'Alcontres 31, 98166 Messina, Italy}
\begin{center}
 \textit{E-mail addresses}: \texttt{lamata@unime.it}; \texttt{mcrupi@unime.it}
\end{center}

}}
\date{}
\maketitle
\Addresses

\begin{abstract}
Let $K$ be a field and $S = K[x_1,\dots,x_n]$ be a polynomial ring over $K$. We discuss the behaviour of the extremal Betti numbers of the class of squarefree strongly stable ideals. More precisely, we give a numerical characterization of the possible extremal Betti numbers (values as well as positions) of such a class of squarefree monomial ideals.
\blfootnote{
\hspace{-0,3cm} \emph{Keywords:} graded ideals, squarefree monomial ideals, minimal graded resolutions.


* \emph{Corresponding author: Marilena Crupi; email: mcrupi@unime.it.}}
\end{abstract}

\section{Introduction }

Let $K$ be a field and $S = K[x_1,\dots,x_n]$ be the polynomial ring in $n$ variables with coefficients in $K$. A squarefree monomial ideal of $S$ is a monomial ideal  generated by squarefree monomials. 
Such ideals are also known as Stanley--Reisner ideals, and quotients by them are called 
Stanley--Reisner rings. The combinatorial nature of these algebraic objects comes from their close
connections to simplicial topology. Many authors have studied the class of squarefree monomial ideals from the viewpoint of commutative algebra and combinatorics (see, for example \cite{AHH2, AHH3, CW, MS}, and the references therein).

Let $I$ be a graded ideal of $S$. A graded Betti number $\beta_{k,k+\ell}(I) \neq 0$ is called {\it extremal} if $\beta_{i,\, i+j}(I) = 0$ for all $i \geq k$, $j \geq \ell$, $(i, j) \neq (k, \ell)$ \cite{BCP}. The pair $(k, \ell)$ is called a \emph{corner} of $I$. If $\beta_{k_i, k_i+\ell_i}(I)$ ($i=1, \ldots, r$) are extremal Betti numbers of a graded ideal $I$, then the set $\C(I) = \{(k_1, \ell_1), (k_2, \ell_2), \ldots, (k_r, \ell_r)\}$ will be called the corner sequence of $I$ \cite[Definition 4.1]{MC}. In the Macaulay or CoCoA Betti diagram of $I$, the graded Betti number $\beta_{i,j}(I)$ is plotted in column $i$ and row $j-i$. Using such a notation, a graded Betti number $\beta_{k,k+\ell}(I)$ is extremal if it is the only entry in the quadrant where it is the northwest corner. 
Projective dimension measures the column index of the easternmost extremal Betti number, whereas regularity measures the row index of the
southernmost extremal Betti number. Indeed, the extremal Betti numbers are a generalization of such meaningful algebraic invariants. 

For a monomial ideal $I$ of $S$, let as denote by $G(I)$ the unique minimal set of monomial generators of $I$ and for a monomial $1 \neq u \in S$, let us define $\supp(u)=\{i: x_i\,\, \textrm{divides}\,\, u\}$. A monomial ideal $I$ of $S$ is \textit{strongly stable} if for all $u \in G(I)$ one has $(x_j u)/x_i \in I$ for all $i \in \supp(u)$ and all $j < i$ \cite{EK, JT}; whereas a squarefree monomial ideal $I$ of $S$ is \textit{squarefree strongly stable} if for all $u \in G(I)$ one has $(x_j u)/x_i \in I$ for all $i \in \supp(u)$ and all $j < i$, $j \notin \supp(u)$ \cite{AHH2,JT}.

Assume that the characteristic of the base field $K$ is zero. If $I$ is a graded ideal of $S$, then the generic initial ideal $\Gin(I)$, with respect to the reverse lexicographical order on $S$ induced by $x_1> \cdots >x_n$, is a strongly stable ideal of $S$  (see, for instance, \cite{Ei, JT}). If $I$ is squarefree, then $\Gin(I)$ is not in general squarefree. In \cite{AHH3}, the authors have introduced a certain operator $\sigma$ which transforms $\Gin(I)$ to a squarefree monomial ideal of $S$. Such an ideal, denoted by $\Gin(I)^{\sigma}$, is squarefree strongly stable \cite[Lemma 1.2.]{AHH3}. On the other hand, \cite[Theorem 2.4.]{AHH3} assures that if $I$ is a squarefree ideal then the extremal Betti numbers are preserved when we pass from $I$ to $\Gin(I)^{\sigma}$. Hence, if one wants to study the extremal Betti numbers of squarefree monomial ideals in a polynomial ring $S = K[x_1,\dots,x_n]$ with $\Char(K)=0$, it is not restrictive to consider the behavior of such extremal Betti numbers for the class of squarefree strongly stable ideals.

In this paper, we are interested to the study of the extremal Betti numbers of the class of squarefree strongly stable ideals of $S$.

The first result on the behavior of the extremal Betti numbers of such a class of squarefree monomial ideals can be found in \cite[Propostion 4.1]{CU2}. More precisely, the authors in \cite{CU2} gave a criterion to determine whether a graded Betti number is extremal: \emph{let $I$ be a  squarefree strongly stable ideal of $S$. $\beta_{k, \, k+ \ell}(I)$ is an extremal Betti number if and only if $k + \ell = \max\{\m(u) : u \in G(I)_{\ell}\}$ and $\m(u)< k+j$, for all $j > \ell$ and for all $u\in G(I)_j$} (Characterization \ref{Char}); $G(I)_{\ell}$ is the set of monomials $u$ of $G(I)$ such that $\deg u = \ell$. They did not give any numerical charaterization of the possible extremal Betti numbers of such a class of ideals. Later, such a criterion was generalized to the class of squarefree strongly stable submodules of a finitely generated graded free $S$--module with a homogeneous basis in \cite[Theorem 4.3]{CF4}. Moreover, a criterion for determining their positions and their number was also given in \cite[Section 5]{CF4}. Such a criterion will be an important tool for the development of this article.

Differently from the non--squarefree case, not much is known about the numerical characterization of the possible extremal Betti numbers (values and positions) of the class of squarefree strongly stable ideals. Indeed, many authors have faced and solved such a question for the class of strongly stable ideals in $S$ (\cite{AC, MC, MC2, MC3, CU1, CU2, HSV}). More precisely, the authors of the previous papers have examined the following problem:
\begin{Probl} \label{probl1} Given two positive integers $n, r$, $1\le r \le n-1$, $r$ pairs of positive integers $(k_1, \ell_1)$,  $\ldots$, $(k_r, \ell_r)$ such that $n-1 \ge k_1 >k_2 > \cdots >k_r\ge 1$ and $1\le \ell_1 < \ell_2 < \cdots < \ell_r$ and $r$ positive integers $a_1, \ldots, a_r$, under which conditions does there exist a graded ideal $I$ of $S=K[x_1,\dots,x_n]$ such that $\beta_{k_1, k_1+\ell_1}(I) = a_1$, $\ldots$, $\beta_{k_r, k_r+\ell_r}(I) = a_r$ are its extremal Betti numbers?
\end{Probl}

Positive answers to Problem \ref{probl1} can be found in \cite[Propositions 2.5, 3.5, Theorem 3.7]{MC}, \cite[Theorem 3.1]{CU1} and \cite[Theorem 3.7]{HSV} when $K$ is a field of characteristic $0$ (see also \cite[Proposition 3.1, Theorem 3.2]{AC}).
More specifically, in all the previous cited papers, numerical characterizations of the possible extremal Betti numbers of a graded ideal $I$ of initial degree $\ge 2$ of a standard graded polynomial ring over a field of characteristic 0  have been given. As we have just underlined, in such a case the generic initial ideal of a graded ideal in $S$ (with respect to the reverse lexicographical order on $S$) is strongly stable and since the extremal Betti numbers are preserved by passing from the graded ideal to its generic ideal \cite{BCP}, the problem is equivalent to the characterization of the possible extremal Betti numbers of a strongly stable ideal of $S$. Moreover, in \cite{AC} a CoCoA package for computing the \emph{smallest strongly stable} ideal of $S$ to face Problem \ref{probl1} has been developed.  In particular, the package is able to determine all the possible $r$-tuples of positive integers $(a_1, \ldots, a_r)$ for which such an ideal does exist. Finally, a complete answer to such a problem reformulated in terms of graded submodules of a finitely generated graded free $S$--module has been stated in \cite[Theorem 4.6]{MC}, \cite[Theorem 4.6]{MC2} and \cite[Theorem 1]{MC3}. 

The purpose of this paper is to numerically characterize the possible extremal Betti numbers of squarefree monomial ideals of a standard graded polynomial ring $S$ over a field of characteristic 0. 
Our techniques involve overall tools from enumerative combinatorics.

The plan of the paper is as follows. In Section \ref{pre}, some notions that  will be used throughout the paper are recalled. In Section \ref{sec:1}, firstly we identify the admissible corner sequences of a squarefree strongly stable ideal for $n=2,3,4$. Then, we determine the maximal number of corners allowed for a squarefree strongly stable ideal $I$ of $S$ with a corner in its initial degree (Propositions \ref{prop:lungh2}, \ref{prop:lungh3}). 
Moreover, given $n-\ell_1$ ($n \geq 5$) pairs of positive integers $(k_1, \ell_1), (k_2, \ell_2), \ldots, (k_{n-\ell_1}, \ell_{n-\ell_1})$, 
with $1\leq k_{n-\ell_1} < k_{n-\ell_1-1} < \cdots < k_1\leq n-3$ and $3 \leq \ell_1 < \ell_2 < \cdots < \ell_{n-\ell_1}\leq n-1$, we determine the conditions under which there exists a squarefree lex ideal  (Definition \ref{def:lex}) $I$ of $K[x_1,\dots,x_n]$ of initial degree $\ell_1$ having $\beta_{k_i, k_i+\ell_i}(I)$, $i=1, \ldots, r$, as extremal Betti numbers (Theorem \ref{thm:lex}). A complete description of the minimal system of monomial generators of $I$ is given. Squarefree lex ideals are a subclass of the class of squarefree strongly stable ideals \cite{AHH2}. Finally, in Section \ref{sec:2}, we face the squarefree version of Problem \ref{probl1}, \emph{i.e.}, the following problem: \emph{Given three positive integers $n\ge 4$, $\ell_1\ge 2$ and $1\le r \le n-\ell_1$, $r$ pairs of positive integers $(k_1, \ell_1)$,  $\ldots$, $(k_r, \ell_r)$ such that $n-3 \ge k_1 >k_2 > \cdots >k_r\ge 2$ and $2\le \ell_1 < \ell_2 < \cdots < \ell_r$, $k_i+\ell_i\le n$ ($i=1, \ldots, r$), and $r$ positive integers $a_1, \ldots, a_r$, under which conditions does there exist a squarefree monomial ideal $I$ of $S=K[x_1,\dots,x_n]$ such that $\beta_{k_1, k_1+\ell_1}(I) = a_1$, $\ldots$, $\beta_{k_r, k_r+\ell_r}(I) = a_r$ are its extremal Betti numbers?} We solve such a problem when $\Char(K)=0$ (Theorem \ref{thm:constru}). In such a case, the question is equivalent to the characterization of the possible extremal Betti numbers of a squarefree strongly stable ideal of $S$ as we have pointed out. The idea behind Theorem \ref{thm:constru} is to establish the bounds for the integers $a_i$ ($i=1, \ldots, r$), starting with $a_r$ and then arriving to $a_1$, by computing the cardinality of suitable sets of monomials. The key result in this Section is Theorem \ref{thm:count}. Let $(k,\ell)$ be a pair of positive integers and let $A^s(k,\ell)$ be the set of all squarefree monomials $u$ of $S$ of degree $\ell$ and such that $\max(u)=k+\ell$, with $\max(u)=\max\{i: x_i\,\, \textrm{divides}\,\, u\}$, ordered by the squarefree lex order $\geq_{\slex}$ defined in Section~\ref{pre}.
If $u\in A^s(k,\ell)$, Theorem \ref{thm:count} shows a method for determining the cardinality of the set of all squarefree monomials $w\in A^s(k,\ell)$ such that $w \geq_{\slex} u$.
We provide some examples illustrating the main obstructions to the issue. 
All the examples are constructed by means of {\em Macaulay2} packages \cite{GDS}, some of which were developed by the authors of this article.

\section{Preliminaries and notation} \label{pre}
Let us consider the polynomial ring $S=K[x_1,\ldots, x_n]$ as an $\NN$-graded ring where $\deg x_i =1$, $i=1,\ldots,n$.
A \textit{monomial ideal} $I$ of $S$ is an ideal generated by monomials. 
If $I$ is a monomial ideal of $S$, we denote by $G(I)$ the unique minimal set of monomial generators of $I$, by $G(I)_{\ell}$ the set of monomials $u$ of $G(I)$ such that $\deg u = \ell$, and by $G(I)_{\ge \ell}$ the set of monomials $u$ of $G(I)$ such that $\deg u \ge \ell$.
If $I=\oplus_{j \geq 0}I_j$ is a graded ideal of the polynomial ring $S$, we denote by $\indeg I$ the \emph{initial degree} of $I$, \emph{i.e.}, the minimum $j$ such that $I_j \neq 0$.

For a monomial $1 \neq u \in S$, we set \[\supp(u)=\{i: x_i\,\, \textrm{divides}\,\, u\},\]
and  we write
\[
\m(u) = \max \{i:i\in \supp(u)\}, \qquad \min(u) =\min\{i : i \in \supp(u)\}.\]
moreover, we set $\m(1) = \min(1) =0$.

A monomial $m \in S$ is called a \textit{squarefree monomial} if $m=x_{i_1}x_{i_2}\cdots x_{i_d}$ with $1\leq i_1<i_2< \cdots < i_d \leq n.$
If $T$ is a subset of $S$, we denote by $\Mon_d(T)$ the set of all monomials in $T$ and by $\Mon_d^s(T)$ the set of all squarefree monomials in $T$.

A monomial ideal $I$ is a \textit{squarefree monomial ideal} if $I$ is a monomial ideal of $S$ generated by squarefree monomials.

\begin{defn} \label{def:strongly} \rm Let $I$ be a squarefree monomial ideal of $S$. $I$  is called a \textit{squarefree stable ideal} if for all $u \in G(I)$ one has $(x_j u)/x_{\m(u)} \in I$ for all $j < \m(u), j \notin \supp(u)$.\\
$I$  is called a \textit{squarefree strongly stable ideal} if for all $u \in G(I)$ one has
$(x_j u)/x_i \in I$ for all $i \in \supp(u)$ and all $j < i$, $j \notin \supp(u)$.
\end{defn}

\begin{rem} 
Let $T$ be a set of squarefree monomials in $S$ of degree $d$. $T$ will be called a \emph{squarefree stable set} if for all $u \in T$ one has $(x_j u)/x_{\m(u)} \in T$ for all $j < \m(u), j \notin \supp(u)$.
$T$ will be called a \emph{squarefree strongly stable set} if for all $u \in T$ one has $(x_j u)/x_i \in T$ for all $i \in \supp(u)$ and all $j < i$, $j \notin \supp(u)$.\\
Hence, a squarefree monomial ideal $I$ of $S$ is squarefree (strongly) stable if $\Mon_d^s(I)$ is a squarefree (strongly) stable set, for all $d$.
\end{rem}

For every $1\leq d\leq n$,
we can order  $\Mon_d^s(S)$ with the \textit{squarefree lexicographic order} $\geq_{\slex}$ \cite{AHH2}.
More precisely, let
\[u=x_{i_1}x_{i_2}\cdots x_{i_d}, \qquad v=x_{j_1}x_{j_2}\cdots x_{j_d},\]
with $1\leq i_1< i_2< \cdots < i_d\leq n$, $1\leq j_1< j_2< \cdots < j_d\leq n$, be squarefree monomials of degree $d$ in $S$, then
\begin{equation}\label{def:slex}
\mbox{$u >_{\textrm{slex}} v$ \qquad if \qquad $i_1=j_1, \ldots, i_{s-1}=j_{s-1}$ \qquad and \qquad $i_s<j_s$},
\end{equation}
for some $1 \leq s \leq d$.

A nonempty set $L \subseteq \Mon_d^s(S)$ is called a \textit{squarefree lexsegment set} of degree $d$ if for $u \in L$, $v \in \Mon_d^s(S)$ such that $v >_{\slex} u$, then $v \in L$.

\begin{defn} \label{def:lex} \rm Let $I$ be a squarefree monomial ideal of $S$.  $I$ is a \textit{squarefree lexsegment ideal} of $S$ if for all squarefree monomials $u\in I$ and all squarefree monomials $v$ of the same degree with $v >_{\slex} u$, it follows that $v \in I$.
\end{defn}

\begin{ex} Let $S = K[x_1, x_2,x_3,x_4,x_5]$. The ideal $I$ $=$ $ (x_1x_2x_3,$ $ x_1x_2x_4,$ $x_1x_2x_5,$ $ x_1x_3x_4,$ $x_2x_3x_4x_5)$  
is a squarefree lexsegment ideal of $S$.
\end{ex}

For any graded ideal $I$ of $S$, there is a minimal graded free $S$-resolution \cite{BH}
\[
\FFF. : 0 \rightarrow F_s \rightarrow \cdots \rightarrow F_1 \rightarrow F_0 \rightarrow I \rightarrow 0,
\]
where $F_i = \oplus_{j \in \ZZ}S(-j)^{\beta_{i,j}}$. The integers $\beta_{i,j} = \beta_{i,j}(I) = \textrm{dim}_K \Tor_i(K,I)_j $ are called the \emph{graded Betti numbers} of $I$.

\begin{defn} {\em \cite{BCP}} \label{def:extr} \rm A graded Betti number $\beta_{k,k+\ell}(I) \neq 0$ is called {\it extremal} if $\beta_{i,\, i+j}(I) = 0$ for all $i \geq k$, $j \geq
\ell$, $(i, j) \neq (k, \ell)$.
\end{defn} 

The pair $(k, \ell)$ is called a \emph{corner} of $I$.

If  $I$ is a squarefree stable ideal, there exists a formula to compute the graded Betti numbers of $I$ (\cite{AHH2}):
\begin{equation}\label{AHHeq}
    \beta_{k, \, k+\ell}(I) =\sum_{u \in G(I)_\ell} \binom{\m(u)-\ell}{k}.
\end{equation}

Because of relation~(\ref{AHHeq}), next characterization holds true \cite{CF4, CU2}.

\begin{Charact}\label{Char}
\label{equiv} Let $I$ be a  squarefree stable ideal of $S$. $\beta_{k, \, k+ \ell}(I)$ is an extremal Betti number if and only if $k + \ell = \max\{\m(u) : u \in G(I)_{\ell}\}$ and $\m(u)< k+j$, for all $j > \ell$ and for all $u\in G(I)_j$.
\end{Charact}

As a consequence of such a characterization, one has that if $I$ is a squarefree stable
ideal of $S$ and $\beta_{k, \, k+ \ell}(I)$ is an extremal Betti number of $I$, then
\begin{equation}\label{eq:extr1}
\beta_{k, \, k+ \ell}(I) = \vert\{u \in G(I)_{\ell}\,:\, \m(u)= k+\ell\}\vert.
\end{equation}
Moreover, setting $\ell = \max\{j:G(I)_j \neq \emptyset\}$, $m = \max\{\m(u)\,:\, u\in G(I)\}$, then
$\beta_{m-\ell,\,m}$ is the unique extremal Betti number of $I$ if and only if
\[m = \max\{\m(u)\,:\, u\in G(I)_{\ell}\},\]
and $\m(w) < m$, for all $w\in G(I)_j$ with $j<\ell$.

\begin{rem}\label{rem:equ} 
If $I$ is a squarefree stable monomial ideal of $S$ and $\beta_{k, k+\ell}(I)$ is an extremal Betti number of $I$, then from Characterization \ref{Char}, we have the following bound:
\begin{equation}\label{diseq1}
    1\leq \beta_{k, k+\ell}(I) \leq \binom{k+\ell-1}{\ell-1}.
\end{equation}
In fact, there exist exactly $\binom{k+\ell-1}{\ell-1}$ squarefree monomials of degree $\ell$ in $S$ with $\m(u) = k+\ell$.
\end{rem}

Now, let $(k_1, \ell_1), \ldots, (k_r, \ell_r)$ ($n-1 \ge k_1 >k_2 > \cdots >k_r\ge 1$, $1\le \ell_1 < \ell_2 < \cdots < \ell_r$) be corners of a graded ideal $I$, according to \cite{MC}, the following notions can be introduced:
\[\C(I) = \{(k_1, \ell_1), \ldots, (k_r, \ell_r)\},\,\,\, a(I) = (\beta_{k_1, k_1+\ell_1}(I), \ldots, \beta_{k_r, k_r+\ell_r}(I)).\]
$\C(I)$ is called the \emph{corner sequence} of $I$, and $a(I)$ the \emph{corner values sequence} of $I$.

If $I$ is a squarefree ideal of $S$, then $k_i + \ell_i \le n$, for all $i=1, \ldots, r$.

\begin{ex} Let $S= K[x_1, x_2, x_3, x_4,x_5, x_6]$ and let
\[I = (x_1x_2, x_1x_3, x_1x_4, x_1x_5, x_2x_3x_4, x_2x_3x_5, x_2x_3x_6, x_2x_4x_5, x_2x_4x_6, x_3x_4x_5x_6)\]
be a squarefree strongly stable ideal of $S$. The extremal Betti numbers of $I$ are $\beta_{3,6}(I)=2,\, \beta_{2,6}(I)=1,$ as the Betti table of $I$ shows:
\[
\begin{array}{cccccc}
    &   & 0 & 1 & 2 & 3 \\
\hline
  2 & : & 4 & 6 & 4 & 1 \\
  3 & : & 5 & 11 & 8 & 2   \\
  4 & : & 1 & 2 & 1 & -
\end{array}
\]
Hence, the corner sequence and the corner values sequence of $I$ are
\[\C(I) = \{(3,3), (2,4)\},\,\,\,\textrm{and}\,\,\, a(I) = (2, 1),\]
respectively.
\end{ex}

We close this Section with some notations from \cite[Section 5]{CF4} that will be useful in the sequel.

Let $I$ be a squarefree stable ideal of $S$. If $I$ is generated in one degree $\ell$, then $I$ has a unique extremal Betti number $\beta_{m-\ell,\,m}(I)$, where $m= \max\{ \m(u)\,:\, u \in G(I)\}$.

Assume $I$ to be generated in degrees $1\leq \ell_1 < \ell_2 < \cdots < \ell_t \leq n$, and denote by $[t]$ the set $\{1, \ldots, t\}$.

Setting
\[m_{\ell_j} = \max\{\m(u)\,:\, u \in G(I)_{\ell_j}\},\]
for $j=1, \ldots, t$, let us consider the following sequence of non negative integers associated to $I$:
\begin{equation}\label{degseq1}
   \bd(I) =(m_{\ell_1}-\ell_1, m_{\ell_2} - \ell_2, \ldots, m_{\ell_t}-\ell_t).
\end{equation}
Such a sequence is called the \textit{degree-sequence} of $I$.

One can observe that, if
\begin{equation}\label{disdegree}
m_{\ell_1}-\ell_1 > m_{\ell_2}-\ell_2> \cdots > m_{\ell_t}-\ell_t,
\end{equation}
then, from Characterization \ref{Char}, $\beta_{m_{\ell_i}-\ell_i,\,m_{\ell_i}}(I)$ is an extremal Betti number of $I$, for $i=1, \ldots, t$. If (\ref{disdegree}) does not hold, one can construct a suitable subsequence of the \textit{degree-sequence} $\bd(I)$, say
\begin{equation}\label{subseq}
\widehat{\bd(I)}=(m_{\ell_{i_1}}-\ell_{i_1}, m_{\ell_{i_2}}-\ell_{i_2}, \ldots, m_{\ell_{i_q}}-\ell_{i_q}),
\end{equation}
with $\ell_1\leq \ell_{i_1} < \ell_{i_2}< \cdots <\ell_{i_q}= \ell_t$, and such that, for $j=1, \ldots, q$, \scalebox{0.96}{$\beta_{m_{\ell_{i_j}}-\ell_{i_j},\,m_{\ell_{i_j}}}(I)$} is an extremal Betti number of $I$.\\
The integer $q\le t$, denoted by $\bl(I)$, and called the \textit{degree-length} of $I$, gives the number of the extremal Betti numbers of the squarefree stable ideal $I$.\\
For more details on this subject see \cite{CF4}.

\section{Extremal Betti numbers of squarefree strongly stable ideals}\label{sec:1}
In this Section, we examine the extremal Betti numbers of squarefree strongly stable ideals in $S=K[x_1, \ldots, x_n]$. More precisely, we identify the admissible corner sequence of a squarefree strongly stable ideal in $S$.

From now on, we assume $\Mon{_\ell}^s(S)$ to be endowed with the squarefree lex order $>_{\slex}$ induced by $x_1>x_2> \cdots >x_n$.

At first, we analyze the simple cases occurring when $n=2, 3$.\\ 
{\bf Case 1.} Let $n=2$ and $S=K[x_1, x_2]$. A squarefree strongly stable ideal $I$ of $S$ can have at most one corner. 
More precisely, $\C(I) = \{(1,1)\}$ with $a(I) = (1)$, \emph{i.e.}, $I = (x_1, x_2)$.\\
{\bf Case 2.} Let $n=3$ and $S=K[x_1, x_2, x_3]$. Also in such a case, a squarefree strongly stable ideal $I$ of $S$ can have at most one corner $(k, \ell)$, $k+\ell \le 3$. Indeed, the only situations that may occur are listed in Table~\ref{tab:1}.

\begin{table}[H]
\[\begin{tabular}{|l|l|l|}
\hline
{\bf Corners} & {\bf Corner values} & {\bf Squarefree strongly stable ideal}\\
\hline
$\C(I) = \{(2,1)\}$  & $a(I) =(1)$ & $I=(x_1, x_2, x_3)$ \\
\hline
$\C(I) = \{(1,1)\}$ & $a(I) =(1)$ & $I=(x_1, x_2)$ \\
\hline
$\C(I) = \{(1,2)\}$ & $a(I)= (1)$ & $I=(x_1x_2, x_1x_3)$\\
\hline
$\C(I) = \{(1,2)\}$ & $a(I)= (2)$ & $I=(x_1x_2, x_1x_3, x_2x_3)$\\
\hline
\end{tabular}
\]
\caption{\label{tab:1}Corner sequences for $n=3$.}
\end{table}

Such easy cases allow us to yield the next result.

\begin{prop}\label{pro:onedegree}
Let $S=K[x_1,\ldots ,x_n]$, $n\ge 2$. If $I$ is a squarefree strongly stable ideal of $S$ with $(k,1)\in \C(I)$, then $\vert \C(I) \vert = 1$. More precisely, $I = (x_1, x_2, \ldots, x_{k+1})$. 
\end{prop}
\begin{proof} First of all one can observe that $G(I)_1 = \{x_1, \ldots, x_{k+1}\}$. If $G(I)_{\ge 2} \neq \emptyset$, then there exists a monomial $u\in G(I)$ of degree $\ell\ge 2$ such that $\m(u) \ge k+2$. A contradiction, since $(k, 1)$ is a corner of $I$.
\end{proof}

Now, let us consider the case $n=4$.\\
{\bf Case 3.} Let $n=4$ and $S=K[x_1, x_2, x_3, x_4]$. Assume $I$ to be a squarefree strongly stable ideal $S$ of initial degree $\ge 2$ (Proposition \ref{pro:onedegree}). Since a pair $(k, \ell ) \in \C(I)$ must satisfy the inequality $k+\ell \le 4$, the situations that can occur in such a case are described in Table~\ref{tab:2}.

\begin{table}[H]
\[
\scalebox{0.9}{
\begin{tabular}{|l|l|l|}
\hline
{\bf Corners} & {\bf Corner values} & {\bf Squarefree strongly stable ideal}\\
\hline
$\C(I) = \{(2,2), (1,3)\}$ & $a(I)$ =(1,1) & $I=(x_1x_2, x_1x_3, x_1x_4, x_2x_3x_4)$ \\
\hline
$\C(I)= \{(1,2)\}$ & $a(I)=(1)$ & $I=(x_1x_2, x_1x_3)$ \\
\hline
$\C(I) = \{(1,2)\}$ & $a(I)= (2)$ & $I=(x_1x_2, x_1x_3, x_2x_3)$\\
\hline
$\C(I) = \{(2,2)\}$ & $a(I)= (1)$ & $I=(x_1x_2, x_1x_3, x_1x_4)$\\
\hline
$\C(I) = \{(2,2)\}$ & $a(I)= (2)$ & $I=(x_1x_2, x_1x_3, x_1x_4, x_2x_3, x_2x_4)$\\
\hline
$\C(I) = \{(2,2)\}$ & $a(I)= (3)$ & $I=(x_1x_2, x_1x_3, x_1x_4, x_2x_3, x_2x_4, x_3x_4)$\\
\hline
$\C(I) = \{(1,3)\}$ & $a(I)= (1)$ & $I=(x_1x_2x_3, x_1x_2x_4)$\\
\hline
$\C(I) = \{(1,3)\}$ & $a(I)= (2)$ & $I=(x_1x_2x_3, x_1x_2x_4, x_1x_3x_4)$\\
\hline
$\C(I) = \{(1,3)\}$ & $a(I)= (3)$ & $I=(x_1x_2x_3, x_1x_2x_4, x_1x_3x_4, x_2x_3x_4)$\\
\hline
\end{tabular}
}\]
\caption{\label{tab:2} Corner sequences for $n=4$.}
\end{table}

\begin{rem}
All the squarefree strongly stable ideals described in Tables~\ref{tab:1} and \ref{tab:2} are the smallest strongly stable ideals with the given data, with respect to the inclusion relation.
\end{rem}

Let $T$ be a subset of $\Mon_d^s(S)$, $d <n$. 
The set of squarefree monomials of degree $d+1$ of $S$
\[
\Shad(T)=\{x_iu :\ u \in T,\ i\notin \supp(u),\ i=1, \ldots, n\}
\]
is called the \textit{squarefree} \textit{shadow} of $T$. Moreover, we define the $i$-th \emph{squarefree} \emph{shadow} recursively by $\Shad^i(T)$ $=$ $\Shad(\Shad^{i-1}(T))$, $i \ge 1$, with $\Shad^0(T)=$ $T$.

Next notion will be crucial for the further developments in this paper.

\begin{defn} \rm Let $u = x_{i_1} \cdots x_{i_q}$ be a squarefree monomial of $S$ of degree $q<n$. We say that $u$ has a $j\gap$ if $i_{j+1}-i_j>1$ for some $1\le j < q$.
The positive integer $i_{j+1}-i_j-1$ will be called the width of the $j\gap$. 
\end{defn}

The $j\gap$ of a squarefree monomial $u=x_{i_1} \cdots x_{i_q}\in S$ will be denoted by $j\gap(u)$, whereas its width will be denoted by $\wdt(j\gap(u))$.  Moreover, we define
\[\Gap(u):=\{j \in [q]: \, \textrm{there exists a}\;  j\gap(u)\}.\]

\begin{defn} \rm A squarefree monomial $u = x_{i_1} \cdots x_{i_q}$ of $S$ will be said gap--free if $\Gap(u)=\emptyset$.
\end{defn}

\begin{ex} Let $S = K[x_1, \ldots, x_{11}]$. The monomial $u=x_1x_3x_4x_6x_{10}\in S$ has three gaps. Indeed, $\Gap(u)=\{1,3,4\}$, $1\gap(u)$, $3\gap(u)$ have both width equal to $1$ and $4\gap(u)$ has width equal to  $3$;
on the contrary, the monomial $v=x_2x_3x_4x_5x_6\in S$ is gap--free. 
\end{ex}

\begin{lem} \label{lem:main} Let $u = x_{i_1} \cdots x_{i_q}$ be a squarefree monomial of degree $q< n-1$ of $S$. Assume $u$ has a gap whose width is $\ge 2$, or  $u$ has at least two gaps.

Then there exist at least two squarefree monomials $v, w\in S$ of degree $q+1$ with $v>_\slex w$, $\m(v) = \m(w)=n$ and such that 
\begin{enumerate}
\item[\rm{(i)}] $v$ is a multiple of $u$; 
\item[\rm{(ii)}] $w$ is not a multiple of $u$.
\end{enumerate}
\end{lem}
\begin{proof} If $\m(u)<n$, we can choose $v= ux_n = x_{i_1} \cdots x_{i_q}x_n$. 
Setting $t= \max \Gap(v)$, the greatest squarefree monomial following $v$ in the squarefree lex order is 
\[\tilde v=x_{i_1} \cdots x_{i_{t-1}}x_{i_t+1}\cdots x_{i_t+q-t+2}.\] 
If $i_t+q-t+2=n$, we choose $w=\tilde v$, otherwise, if $i_t+q-t+2 < n$, we choose $w = x_{i_1} \cdots x_{i_{t-1}} x_{i_t+1}\cdots x_{i_t+q-t+1}x_n$.
Finally, $v>_\slex w$, $u \mid v$ and $u \nmid w$. Note that $t\le q$.

Now, assume $\m(u)=n$. If $t= \max \Gap(u)$, let
\[v=  x_{i_1}\cdots x_{i_t}x_{i_{t+1}-1}x_{i_{t+1}} \cdots x_{i_{q-1}}x_{i_q} = x_{i_1}\cdots x_{i_t}x_{i_{t+1}-1}x_{i_{t+1}} \cdots x_{i_{q-1}}x_n.\]
Furthermore, if $p =  \max \Gap(v)$,  then the greatest squarefree monomial following $v$ in the squarefree lex order is  
\[
\tilde v=x_{i_1} \cdots x_{i_{p-1}}x_{i_p+1}\cdots x_{i_p+q-p+2}.\]
Hence, if $i_p+q-p+2= n$, we choose $w=\tilde v$, otherwise, if $i_p+q-p+2< n$, we choose 
$w$ $=$ $x_{i_1} \cdots x_{i_{p-1}}x_{i_p+1}\cdots x_{i_p+q-p+1}x_n$.

Note that the assumption on the gaps of the squarefree monomial $u$ assures us that we can construct both the monomials $v$ and $w$. 
\end{proof}

Next results easily follow.
\begin{prop} 
\label{prop:lungh2} Let $I$ be a squarefree strongly stable ideal of $S = K[x_1,\dots,x_n]$, $n\ge 4$, with initial degree $2$ and with a corner in degree $2$. Then
\begin{enumerate}
\item[\rm{(1)}] $I$ has at most $n-2$ corners for $n=4$;
\item[\rm{(2)}] $I$ has at most $n-3$ corners for $n\ge 5$.
\end{enumerate}
\end{prop}
\begin{proof} (1). It follows from Case 3. \\
(2). Let $n\ge 5$. An admissible degree--sequence of $I$ is the following one
\[\bd(I) = (n-2, n-3, \cdots, n-(n-2)=2).\]
Indeed, setting $w_1 = x_1x_n$, since $1\gap(w_1)$ has width $n-2$, then Lemma \ref{lem:main} assures that there exist at least $n-4$ squarefree monomials $w_2, \ldots, w_{n-3}$ in $S$ of degrees $3, \ldots, n-2$, respectively, with 
$\m(w_i)=n$, and $n-4$ squarefree monomials $v_2, \ldots, v_{n-3}$ of degrees $3, \ldots, n-2$, respectively, with $\m(v_i) = n$ and such that $v_i >_\slex w_i$, $w_{i-1}\mid v_i$, $v_i\nmid w_i$, for $i=2, \ldots, n-3$. Using the same techniques as in Lemma \ref{lem:main}, one can easily verify that $w_i \nmid w_{i+1}$ ($i=1, \ldots, n-4$). 

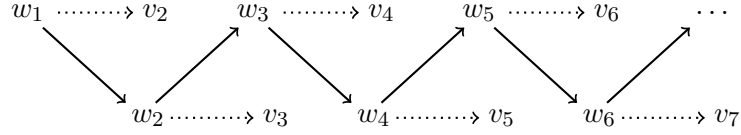
\begin{center}
\begin{tikzpicture}
  \foreach \x in {0,1,2}{
    \draw[thick,->] (\x*2+\x,2)--(\x*2+\x+1.1,1);
    \draw[thick,->] (\x*2+\x+1.5,1)--(\x*2+\x+2.6,2);       
    \draw[thick,dotted,->] (\x*2+\x+0.2,2.2)--(\x*2+\x+1.2,2.2);
    \draw[thick,dotted,->] (\x*2+\x+1.7,0.8)--(\x*2+\x+2.8,0.8);
    \pgfmathtruncatemacro\y{\x*2+1}    
    \draw (\x*2+\x-0.2,2.2)node{$w_\y$};
    \pgfmathtruncatemacro\y{\x*2+2}
    \draw (\x*2+\x+1.5,2.2)node{$v_\y$};
    \draw (\x*2+\x+1.4,0.8)node{$w_\y$};
    \pgfmathtruncatemacro\y{\x*2+3}
    \draw (\x*2+\x+3.1,0.8)node{$v_\y$};
  }
  \draw (8.9,2.2)node{$\ldots$};
\end{tikzpicture}
\end{center}
The monomials $w_i$ ($i=1, \ldots, n-3$) will be called \emph{basic monomials}. 

Next tables list the basic monomials for $n=5,\ldots, 9$. For $n\ge 10$, the construction of such elements proceeds smoothly.  
\begin{figure}[H]
\centering
\subfloat[]{
$
\begin{array}{|cc|}
\hline
\multicolumn{2}{|c|}{\mathbf{n=5}}\\
\hline
 v_i & w_i \\
\hline
             & \mathbf{x_1x_5} \\
  x_1x_4x_5  & \mathbf{x_2x_3x_5} \\
  x_2x_3x_4x_5  & \mathbf{-} \\
  & \\
  \hline
\end{array}
$}
\hspace{.5cm}
\subfloat[]{
$
\begin{array}{|cc|}
\hline
\multicolumn{2}{|c|}{\mathbf{n=6}}\\
\hline
 v_i & w_i \\
\hline
             & \mathbf{x_1x_6} \\
  x_1x_5x_6  & \mathbf{x_2x_3x_6} \\
  x_2x_3x_5x_6  & \mathbf{x_2x_4x_5x_6} \\
  x_2x_3x_4x_5x_6  & \mathbf{-}\\
\hline
\end{array}
$}
\hspace{.5cm}
\subfloat[]{
$
\begin{array}{|cc|}
\hline
\multicolumn{2}{|c|}{\mathbf{n=7}}\\
\hline
 v_i & w_i \\
\hline
             & \mathbf{x_1x_7} \\
  x_1x_6x_7  & \mathbf{x_2x_3x_7} \\
  x_2x_3x_6x_7  & \mathbf{x_2x_4x_5x_7} \\
  x_2x_4x_5x_6x_7  & \mathbf{x_3x_4x_5x_6x_7}\\
\hline
\end{array}
$}
\hspace{.5cm}
\subfloat[]{
$
\begin{array}{|cc|}
\hline
\multicolumn{2}{|c|}{\mathbf{n=8}}\\
\hline
 v_i & w_i \\
\hline
             & \mathbf{x_1x_8} \\
  x_1x_7x_8  & \mathbf{x_2x_3x_8} \\
  x_2x_3x_7x_8  & \mathbf{x_2x_4x_5x_8} \\
  x_2x_4x_5x_7x_8  & \mathbf{x_2x_4x_6x_7x_8}\\
  x_2x_4x_5x_6x_7x_8  & \mathbf{x_3x_4x_5x_6x_7x_8}\\
  & \\
  \hline
\end{array}
$}
\hspace{.5cm}
\subfloat[]{
$
\begin{array}{|cc|}
\hline
\multicolumn{2}{|c|}{\mathbf{n=9}}\\
\hline
 v_i & w_i \\
\hline
             & \mathbf{x_1x_9} \\
  x_1x_8x_9  & \mathbf{x_2x_3x_9} \\
  x_2x_3x_8x_9  & \mathbf{x_2x_4x_5x_9} \\
  x_2x_4x_5x_8x_9  & \mathbf{x_2x_4x_6x_7x_9}\\
  x_2x_4x_6x_7x_8x_9  & \mathbf{x_2x_5x_6x_7x_8x_9}\\
  x_2x_4x_5x_6x_7x_8x_9  & \mathbf{x_3x_4x_5x_6x_7x_8x_9}\\
\hline  
\end{array}
$}
\end{figure}

Note that the construction of the basic elements ends up as soon as one gets a gap--free monomial.
\end{proof}

\begin{ex} \label{expl:8} Let $S= K[x_1, x_2, x_3, x_4,x_5, x_6, x_7, x_8]$, and let
\begin{align*}
I = ( & x_1x_2, x_1x_3, x_1x_4, x_1x_5, x_1x_6, x_1x_7, x_1x_8, x_2x_3x_4, x_2x_3x_5, x_2x_3x_6, x_2x_3x_7,\\
& x_2x_3x_8, x_2x_4x_5x_6, x_2x_4x_5x_7, x_2x_4x_5x_8, x_2x_4x_6x_7x_8, x_3x_4x_5x_6x_7x_8)
\end{align*}
be a squarefree strongly stable ideal of $S$.
The \textit{degree-sequence} of $I$ is
\begin{equation*}
\bd(I)=(m_2-2,m_3-3,m_4-4,m_5-5,m_6-6)=(6,5,4,3,2).
\end{equation*}
$I$ has initial degree $2$ and $\bl(I)= 5$. The extremal Betti numbers of $I$ are\\ $\beta_{8-2,8}(I)=$ $\beta_{8-3,8}(I)$ $=\beta_{8-4,8}(I)=$ $\beta_{8-5,8}(I)=$ $\beta_{8-6,8}(I)=1$, as the Betti table of $I$ shows:
\[
\begin{array}{ccccccccc}
    &   & 0 & 1 & 2 & 3 & 4 & 5 & 6\\
\hline
  2 & : & 7 & 21 & 35 & 35 & 21 & 7 & 1 \\
  3 & : & 5 & 15 & 20 & 15 & 6 & 1 & - \\
  4 & : & 3 & 9 & 10 & 5 & 1 & - & - \\
  5 & : & 1 & 3 & 3 & 1 & - & - & - \\
  6 & : & 1 & 2 & 1 & - & - & - & - 
\end{array}
\]
\end{ex}

\begin{prop}
\label{prop:lungh3} Let $n\geq 5$ and let $I$ be a squarefree strongly stable ideal of $S=K[x_1,\dots,$ $x_n]$ with initial degree $\ell \geq 3$ and with a corner in degree $\ell$.
Then $I$ has at most $n-\ell$ corners.
\end{prop}
\begin{proof} Using the same reasoning as in Proposition \ref{prop:lungh2}, an admissible 
degree--sequence of $I$ is the following one:
\[\bd(I) = (n-\ell, n-(\ell+1), \cdots, n-(n-1)=1),\]
with $\bl(I)=n-\ell$.

Next tables show the basic monomials for $n=5, \ldots, 8$ and $\ell =3$. For $n\ge 8$ ($\ell =3$), the construction of such elements proceeds smoothly.
\begin{figure}[H]
\centering
\subfloat[]{
$
\begin{array}{|cc|}
\hline
\multicolumn{2}{|c|}{\mathbf{n=5}}\\
\hline
 v_i & w_i \\
\hline
             & \mathbf{x_1x_2x_5} \\
  x_1x_2x_4x_5  & \mathbf{x_1x_3x_4x_5} \\
  x_1x_2x_3x_4x_5  & \mathbf{-} \\
  & \\
  \hline
\end{array}
$}
\hspace{.5cm}
\subfloat[]{
$
\begin{array}{|cc|}
\hline
\multicolumn{2}{|c|}{\mathbf{n=6}}\\
\hline
 v_i & w_i \\
\hline
             & \mathbf{x_1x_2x_6} \\
  x_1x_2x_5x_6  & \mathbf{x_1x_3x_4x_6} \\
  x_1x_3x_4x_5x_6  & \mathbf{x_2x_3x_4x_5x_6} \\
  x_1x_2x_3x_4x_5x_6  & \mathbf{-}\\
\hline
\end{array}
$}
\end{figure}
\begin{figure}{H}
\centering
\subfloat[]{
$
\begin{array}{|cc|}
\hline
\multicolumn{2}{|c|}{\mathbf{n=7}}\\
\hline
 v_i & w_i \\
\hline
             & \mathbf{x_1x_2x_7} \\
  x_1x_2x_6x_7  & \mathbf{x_1x_3x_4x_7} \\
  x_1x_3x_4x_6x_7  & \mathbf{x_1x_3x_5x_6x_7} \\
  x_1x_3x_4x_5x_6x_7  & \mathbf{x_2x_3x_4x_5x_6x_7}\\
  & \\
\hline
\end{array}
$}
\hspace{.5cm}
\subfloat[]{
$
\begin{array}{|cc|}
\hline
\multicolumn{2}{|c|}{\mathbf{n=8}}\\
\hline
 v_i & w_i \\
\hline
             & \mathbf{x_1x_2x_8} \\
  x_1x_2x_7x_8  & \mathbf{x_1x_3x_4x_8} \\
  x_1x_3x_4x_7x_8  & \mathbf{x_1x_3x_5x_6x_8} \\
  x_1x_3x_5x_6x_7x_8  & \mathbf{x_1x_4x_5x_6x_7x_8}\\
  x_1x_3x_4x_5x_6x_7x_8  & \mathbf{x_2x_3x_4x_5x_6x_7x_8}\\
  \hline
\end{array}
$}
\end{figure}
Also in this case, the construction of the basic elements ends up as soon as one gets a gap--free monomial.
\end{proof}

\begin{ex} Let $S= K[x_1, x_2, x_3, x_4,x_5, x_6, x_7, x_8]$ and let
\begin{align*}
I = ( & x_1x_2x_3, x_1x_2x_4, x_1x_2x_5, x_1x_2x_6, x_1x_2x_7, x_1x_2x_8, x_1x_3x_4x_5, x_1x_3x_4x_6,\\
& x_1x_3x_4x_7, x_1x_3x_4x_8, x_1x_3x_5x_6x_7, x_1x_3x_5x_6x_8, x_1x_4x_5x_6x_7x_8,\\ & x_2x_3x_4x_5x_6x_7x_8)
\end{align*}
be a squarefree strongly stable ideal of $S$ initial degree $3$.
The \textit{degree-sequence} of $I$ is
\begin{equation*}
    \bd(I)=(m_2-3,m_3-4,m_4-5,m_5-6,m_6-7)=(5,4,3,2,1).
\end{equation*}

\noindent The extremal Betti numbers of $I$ are $\beta_{8-3,8}(I)=$ $\beta_{8-4,8}(I)$ $=$ $\beta_{8-5,8}(I)=$ $\beta_{8-6,8}(I)$ $=$ $\beta_{8-7,8}(I)=1$, as the Betti table of $I$ shows
\[
\begin{array}{cccccccc}
    &   & 0 & 1 & 2 & 3 & 4 & 5\\
\hline
  3 & : & 6 & 15 & 20 & 15 & 6 & 1 \\
  4 & : & 4 & 10 & 10 & 5 & 1 & - \\
  5 & : & 2 & 5 & 4 & 1 & - & - \\
  6 & : & 1 & 2 & 1 & - & - & - \\
  7 & : & 1 & 1 & - & - & - & - 
\end{array}
\]
\end{ex}

The next example considers a squarefree monomial ideal $I$ of $S$ without a corner in its initial degree, and shows the construction of a squarefree monomial ideal $J$ of $S$ with a corner in its initial degree and with the same extremal Betti numbers (positions and values) of $I$.

\begin{ex}  Consider the following monomial ideal $I$ of $S= K[x_1,\dots,x_5]$:
\[
I=(x_1x_2, x_1x_3x_4, x_1x_3x_5, x_2x_3x_4x_5).
\]
$I$ is squarefree strongly stable of initial degree $2$ and with $\C(I)$ $=$ $\{(2,3),$ $(1,4)\}$. From the Betti table of $I$, one can note that there is no corner in its initial degree:

\begin{figure}[H]
\centering
\subfloat[]{
$
\begin{array}{ccccc}
    &   & 0 & 1 & 2\\
\hline
  2 & : & 1 & - & -\\
  3 & : & 2 & 3 & 1\\
  4 & : & 1 & 1 & -\\
\end{array}
$}
\caption{Betti Table of $I$}
\end{figure}

Furthermore, we can construct a squarefree strongly stable ideal $J$ in $S$ with initial degree $3$ and $\C(J)=\{(2,3),(1,4)\}$. It is
\[
J=(x_1x_2x_3, x_1x_2x_4, x_1x_2x_5, x_1x_3x_4x_5).
\]

Note that $J$ is the smallest squarefree strongly stable ideal of $S$ with corner sequence $\{(2,3),(1,4)\}$:
\begin{figure}[H]
\centering
\subfloat[]{
$
\begin{array}{ccccc}
    &   & 0 & 1 & 2\\
\hline
  3 & : & 3 & 3 & 1\\
  4 & : & 1 & 1 & -\\
  & & & &
\end{array}
$}
\vspace{-.5cm} 
\caption{Betti Table of $J$}
\end{figure}

\end{ex}

\begin{rem}  It is worthy to point out that a squarefree strongly stable ideal $I$ of $S=K[x_1, \ldots, x_n]$ ($n\ge 5$) of initial degree $\ell\ge 2$ with a corner in degree $\ell$ and such that
\begin{align*}
\bd(I)&= (n-2, n-3, \ldots, 2), \,\, \textrm{for} \,\, \ell=2,\\[5pt]
\bd(I)&= (n-\ell, n-\ell-1, \ldots, 1), \,\, \textrm{for} \,\, \ell \ge 3
\end{align*}
is a squarefree lex ideal of $S$.

Hence, one can observe that a squarefree lex ideal of the polynomial ring $S$ of initial degree $\ell\ge 2$ and with a corner in degree $\ell$ can have at most $n-\ell$ corners unlike the non--squarefree case. Indeed, a lex ideal $I$ of a polynomial ring can have at most $2$ corners \cite[Theorem 3.2]{CU1} (see also \cite[Proposition 2.1]{CU2}).
\end{rem}

For $u, v \in \Mon_d^s(S)$, $u \ge_\slex v$, let us define the following set of squarefree monomials:
\[\mathcal{L}(u, v) = \{z \in \Mon_d^s(S) : u \ge_\slex z \ge_\slex v \}.\]

\begin{thm}
\label{thm:lex}  Let $n \geq 5$ and $\ell_1\ge 3$ two integers.  Given $n-\ell_1$ pairs of positive integers 
\begin{equation}\label{eq:pairs}
(k_1, \ell_1), (k_2, \ell_2), \ldots, (k_{n-\ell_1}, \ell_{n-\ell_1}),
\end{equation}
with $1\leq k_{n-\ell_1} < k_{n-\ell_1-1} < \cdots < k_1\leq n-3$ and $3 \leq \ell_1 < \ell_2 < \cdots < \ell_{n-\ell_1}\leq n-1$, 
then there exists a squarefree lex ideal $I$ of $S$ of initial degree $\ell_1$ and with the pairs in (\ref{eq:pairs}) as corners if and only if $k_i+\ell_i=n$, for $i=1,\ldots, n-\ell_1$. 
\end{thm}
\begin{proof} Set $S = K[x_1,\dots,x_n]$. If there exists a squarefree lex ideal $I$ of $S$ of initial degree $\ell_1$ and with the pairs in (\ref{eq:pairs}) as corners, then Proposition \ref{prop:lungh3} forces that $k_i+\ell_i=n$, for $i=1,\ldots, n-\ell_1$. 

Conversely, assume there exist $n-\ell_1$ pairs of positive integers 
\begin{equation}
(k_1, \ell_1), (k_2, \ell_2), \ldots, (k_{n-\ell_1}, \ell_{n-\ell_1}),
\end{equation}
with $1\leq k_{n-\ell_1} < k_{n-\ell_1-1} < \cdots < k_1\leq n-3$, $3 \leq \ell_1 < \ell_2 < \cdots < \ell_{n-\ell_1}\leq n-1$ and $k_i+\ell_i=n$, for $i=1,\ldots, n-\ell_1$.\\ 
We prove that there exists a squarefree lex ideal $I$ of $S$ generated in degrees $\ell_1, \ell_2, \ldots, \ell_{n-\ell_1}$ with $\C(I) = \{(k_1,\ell_1), \ldots, (k_{n-\ell_1}, \ell_{n-\ell_1})\}$.

Setting $s = \max\{i: \ell_1+2i-3 \le n-2\}$, the required monomial ideal $I$ can be constructed as follows.\\
{\bf Step 1.} For $i=1, \ldots, s$, let
\begin{enumerate}
\item[-] $G(I)_{\ell_1} = \mathcal{L}\left(u_1, v_1\right)$, with $u_1= x_1x_2\cdots x_{\ell_1}$ and $v_1 = x_1x_2\cdots x_{\ell_1-1}x_n$;
\item[-] $G(I)_{\ell_i} = G(I)_{\ell_1+i-1}= \mathcal{L}\left(u_i, v_i\right)$, with
\begin{align*}
u_i &=x_1x_2\cdots x_{\ell_1-2}\prod_{j=0}^{i-2}x_{\ell_1+2j}x_{\ell_1+2(i-2)+1}x_{\ell_1+2(i-2)+2} \\
       &=x_1x_2\cdots x_{\ell_1-2}\prod_{j=0}^{i-2}x_{\ell_1+2j}x_{\ell_1+2i-3}x_{\ell_1+2i-2}
\end{align*}
and

\[
v_i =x_1\cdots x_{\ell_1-2}\prod_{j=0}^{i-2}x_{\ell_1+2j}x_{\ell_1+2(i-2)+1}x_n =
       x_1\cdots x_{\ell_1-2}\prod_{j=0}^{i-2}x_{\ell_1+2j}x_{\ell_1+2i-3}x_n.
\]
\end{enumerate}
\vspace{0,3cm}
{\bf Step 2.} Let us consider the squarefree monomial 
\[v_s = x_1x_2\cdots x_{\ell_1-2}\prod_{j=0}^{s-2}x_{\ell_1+2j}x_{\ell_1+2s-3}x_n.\]
Since, $\ell_1+2s-3 \le n-2 $, the smallest monomial belonging to the $\Shad(G(I)_{\ell_s})$ is
\[w_{s+1} = x_1x_2\cdots x_{\ell_1-2}\prod_{j=0}^{s-2}x_{\ell_1+2j}x_{\ell_1+2s-3}x_{n-1}x_n.\]
We distinguish two cases: $\ell_1+2s-3 = n-2$, and $\ell_1+2s-3 < n-2$.\\
{\bf Claim 1}. If $\ell_1+2s-3 < n-2$, then $\ell_1+2s-3 = n-3$.\\
Indeed, by the meaning of $s$, $\ell_1+2(s+1)-3 \ge n-1$. Hence, 
$\ell_1+2s-3 \ge n-3$ and
\[n-3 \le \ell_1+2s-3  < n-2\]
and consequently $\ell_1+2s-3 =n-3$. The claim follows.\\

Let us consider $\ell_1+2s-3 = \ell_1+2(s-2)+1= n-2$. In such a case,
\begin{align*}
w_{s+1} &= x_1\cdots x_{\ell_1-2}\prod_{j=0}^{s-2}x_{\ell_1+2j}x_{\ell_1+2(s-2)+1}x_{n-1}x_n\\
 &= x_1\cdots x_{\ell_1-2}\prod_{j=0}^{s-3}x_{\ell_1+2j}x_{n-3}x_{n-2}x_{n-1}x_n.
\end{align*}
Hence, the greatest squarefree monomial of $S$ following $w_{s+1}$ is
\[u_{s+1} =x_1x_2\cdots x_{\ell_1-2}\prod_{j=0}^{s-4}x_{\ell_1+2j}x_{\ell_1+2(s-3)+1}x_{\ell_1+2(s-3)+2}\cdots x_{\ell_1+2(s-3)+5}.\]
Note that $\m(u_{s+1}) = \ell_1+2(s-3)+5 = \ell_1+2s-3+2= n-2+2 =n$, whereupon we choose
\[G(I)_{\ell_{s+1}} =  \{u_{s+1}\}.\]
The smallest squarefree monomial belonging to $\Shad(G(I)_{\ell_{s+1}})$ is
\begin{align*}
w_{s+2} &= x_1x_2\cdots x_{\ell_1-2}\prod_{j=0}^{s-4}x_{\ell_1+2j}x_{\ell_1+2(s-3)}x_{\ell_1+2(s-3)+1}x_{\ell_1+2(s-3)+2}\cdots x_{\ell_1+2(s-3)+5}\\
&=x_1x_2\cdots x_{\ell_1-2}\prod_{j=0}^{s-4}x_{\ell_1+2j}x_{n-5}x_{n-4}x_{n-3}x_{n-2}x_{n-1}x_n.
\end{align*}
Therefore, the greatest squarefree monomial of $S$ following $w_{s+2}$ is
\[u_{s+2} =x_1x_2\cdots x_{\ell_1-2}\prod_{j=0}^{s-5}x_{\ell_1+2j}x_{\ell_1+2(s-4)+1}x_{\ell_1+2(s-4)+2}\cdots x_{\ell_1+2(s-4)+7}.\]
Note that $\m(u_{s+2}) = \ell_1+2(s-4)+7 = \ell_1+2s-3+2= n-2+2 =n$. Thus, we choose
\[G(I)_{\ell_{s+2}} =  \{u_{s+2}\},\]
and so on. In general, 
\[G(I)_{\ell_{s+q}} =  \{u_{s+q}\},\] 
with
\begin{equation*}
\resizebox{\hsize}{!}{$\displaystyle
u_{s+q} =x_1x_2\cdots x_{\ell_1-2}\prod_{j=0}^{s-2-(q+1)}x_{\ell_1+2j}x_{\ell_1+2(s-2-q)+1}x_{\ell_1+2(s-2-q)+2}\cdots x_{\ell_1+2(s-2-q)+2q+3},$}
\end{equation*}
for $q=1,\ldots, t$, where $t$ is the positive integer such that $s-2-(t+1)=0$. It is easy to verify that $\m(u_{s+q}) = n$.\\
{\bf Claim 2.} $s+t = n-\ell_1-2$.\\
Since, $\m(u_{s+t}) = n$, and $t +1= s-2$ ($t = s-3$), then 
\[n= \ell_1+2(s-2-t)+2t+3 = \ell_1+2(t+1-t)+2t+3 = \ell_1+2t+5.\]
Hence,
\[n-\ell_1-2 = \ell_1+2t+5 -\ell_1-2 =2t+3 = 2s-3 = s+t.\]
The claim follows.

Finally, we choose
\[
G(I)_{\ell_{n-\ell_1-1}} = G(I)_{s+t+1}=  \{u_{s+t+1}\} =\{x_1x_2\cdots x_{\ell_1-2}x_{\ell_1+1}\cdots x_{n}\},\]
\[
G(I)_{\ell_{n-\ell_1}} = G(I)_{s+t+2} =  \{u_{s+t+2}\} = \{x_1x_2\cdots x_{\ell_1-3}x_{\ell_1-1}x_{\ell_1}\cdots x_{n}\}.
\]
\par\bigskip\bigskip

Now, let us consider the case $\ell_1+2s-3 =n-3$. 
In such a case, the smallest monomial belonging to $\Shad(G(I)_{\ell_s})$ is
\begin{align*}
w_{s+1} &=x_1x_2\cdots x_{\ell_1-2}\prod_{j=0}^{s-2}x_{\ell_1+2j}x_{\ell_1+2(s-2)+1}x_{n-1}x_n\\
       &=x_1x_2\cdots x_{\ell_1-2}\prod_{j=0}^{s-2}x_{\ell_1+2j}x_{n-3}x_{n-1}x_n.
\end{align*}
Therefore, the greatest squarefree monomial of $S$ following $w_{s+1}$ is 
\[u_{s+1} =x_1x_2\cdots x_{\ell_1-2}\prod_{j=0}^{s-2}x_{\ell_1+2j}x_{n-2}x_{n-1}x_n.\]
Since $\m(u_{s+1}) = n$, we choose
\[G(I)_{\ell_{s+1}} =  \{u_{s+1}\}.\]
By hypothesis, $\ell_1+2(s-2) = n-4$, so that the smallest squarefree monomial belonging to 
$\Shad(G(I)_{\ell_{s+1}})$ is 
\begin{align*}
w_{s+2} &= x_1\cdots x_{\ell_1-2}\prod_{j=0}^{s-2}x_{\ell_1+2j}x_{n-3}x_{n-2}x_{n-1}x_n\\
&=x_1\cdots x_{\ell_1-2}\prod_{j=0}^{s-3}x_{\ell_1+2j}x_{n-4}x_{n-3}x_{n-2}x_{n-1}x_n.
\end{align*}
Consequently, the greatest squarefree monomial of $S$ following $w_{s+2}$ is 
\[u_{s+2} =x_1x_2\cdots x_{\ell_1-2}\prod_{j=0}^{s-4}x_{\ell_1+2j}x_{\ell_1+2(s-3)+1}x_{\ell_1+2(s-3)+2}\cdots x_{\ell_1+2(s-3)+6}.\]
Note that $\m(u_{s+2}) = \ell_1+2(s-3)+6 = \ell_1+2s = n$, whence we choose
\[G(I)_{\ell_{s+2}} =  \{u_{s+2}\}.\]
In general, 
\[G(I)_{\ell_{s+q}} =  \{u_{s+q}\},\] 
with
\[u_{s+q} =x_1x_2\cdots x_{\ell_1-2}\prod_{j=0}^{s-2-q}x_{\ell_1+2j}x_{\ell_1+2(s-2-(q-1))+1}\cdots x_{\ell_1+2(s-2-(q-1))+2q+2},\]
for $q=1,\ldots, t$, where $t$ is the positive integer such that $s-2-t=0$ ($t=s-2$). It is easy to verify that $\m(u_{s+q}) = n$.\\ 

Also in such a case we can verify that $s+t = n-\ell_1-2$. Indeed,
since $\m(u_{s+t}) = n$, and $t = s-2$, then 
\[n= \ell_1+2(s-2-(t-1))+2t+2 = \ell_1+2t+4,\]
and
\[n-\ell_1-2 = \ell_1+2t+4 -\ell_1-2 =2t+2 = 2(s-2)+2 = s+t.\]

Finally, as in the previous case, we can choose
\[G(I)_{\ell_{n-\ell_1-1}} = G(I)_{s+t+1}=  \{x_1x_2\cdots x_{\ell_1-2}x_{\ell_1+1}\cdots x_{n}\},\]
and
\[G(I)_{\ell_{n-\ell_1}} = G(I)_{s+t+2} =\{x_1x_2\cdots x_{\ell_1-3}x_{\ell_1-1}x_{\ell_1}\cdots x_{n}\}.\]

It is worthy observing that $I$ is the smallest squarefree lex ideal  of $S$ with $\C(I)$ $=$  $\{(k_1, \ell_1), (k_2, \ell_2), \ldots, (k_r, \ell_r)\}$
and such that $\beta_{k_i, \,k_i+\ell_i}(I)=1$, for all $i$, \emph{i.e.},  $a(I) = (1, \ldots, 1)$. 
\end{proof}

\section{A numerical characterization of extremal Betti numbers}\label{sec:2}

In this Section, we face the following problem.
\begin{Probl} \label{probl:1}
Given three positive integers $n\ge 4$, $\ell_1\ge 2$ and $1\le r \le n-\ell_1$, $r$ pairs of positive integers $(k_1, \ell_1)$,  $\ldots$, $(k_r, \ell_r)$ such that $n-3 \ge k_1 >k_2 > \cdots >k_r\ge 2$ and $2\le \ell_1 < \ell_2 < \cdots < \ell_r$, $k_i+\ell_i\le n$ ($i=1, \ldots, r$), and $r$ positive integers $a_1, \ldots, a_r$, under which conditions does there exist a squarefree monomial ideal $I$ of $S=K[x_1,\dots,x_n]$ such that $\beta_{k_1, k_1+\ell_1}(I) = a_1$, $\ldots$, $\beta_{k_r, k_r+\ell_r}(I) = a_r$ are its extremal Betti numbers?
 \end{Probl}
 
For a pair of positive integers $(k, \ell)$ such that $k+\ell\le n$, we define the following set:
\[A^s(k, \ell) = \{u \in \Mon^s_{\ell}(S) : \m(u) = k+\ell\}.\]

Setting $A^s(k, \ell) = \{u_1, \ldots, u_q\}$, we can suppose, possibly 
after a permutation of the indices, that
\begin{equation} \label{vert}
u_1 >_{\slex} u_2 >_{\slex} \cdots   >_{\slex} u_q.
\end{equation}
For the $i$-th monomial $u$ of degree $\ell$ with $\m(u) = k + \ell$, we mean the monomial of $A^s(k, \ell)$ that appears in the $i$-th position of (\ref{vert}), for $1\le i \le q$.  Note that $u_1 = x_1x_2\cdots x_{\ell-1}x_{k+\ell}$, $u_q = x_{k+1}\cdots x_{k+\ell}$, and $q = \vert A^s(k, \ell)\vert= \binom{k+\ell -1}{\ell -1}$. 

Furthermore, if $u_i, u_j$, $i<j$, are two monomials in (\ref{vert}), we define the following subsets of $A^s(k, \ell)$:
\[[u_i, u_j] = \{w \in  A^s(k, \ell) : u_i \ge_{\slex} w \ge_{\slex} u_j\},\]
\[[u_i, u_j) = \{w \in  A^s(k, \ell) : u_i \ge_{\slex} w >_{\slex} u_j\};\]
$[u_i, u_j]$ will be called the \emph{segment} of $A^s(k, \ell)$ of initial element $u_i$ and final element $u_j$, whereas $[u_i, u_j)$ will be called the \emph{left segment} of $A^s(k, \ell)$ of initial element $u_i$ and final element $u_j$. 
If $i=j$, we set $[u_i,u_j] = \{u_i\}$.

\begin{rem} From (\ref{eq:extr1}), if $(k, \ell)$ is a corner of a squarefree stable ideal $I$ and $\beta_{k, k+\ell}(I)= a$, then there exists a segment $[v_1, v_a]$ of  
 $A^s(k, \ell)$ such that $a = \vert [v_1, v_a]\vert$.
\end{rem}

Next lemma will be crucial in the sequel. It can be easily proved by induction on $n$.

\begin{lem}\label{lem:count} Let $n$ and $q\ge 1$ be two positive integers such that $n\ge q$. Then
\[\binom n q = \binom {n-1} {q-1} + \binom {n-2} {q-1}+ \cdots +\binom {q-1} {q-1}.\]
\end{lem}

Given a monomial $u \in A^s(k, \ell)$, the next proposition shows a method, involving Lemma~\ref{lem:count}, to count the number of monomials $v\in A^s(k, \ell)$ such that $v\ge_{slex} u$. 

\begin{thm}\label{thm:count} Let $(k,\ell)$ be a pair of positive integers with $\ell\geq 2$ and let 
$u=x_{i_1}x_{i_2}\cdots x_{i_{\ell-1}} x_{i_{\ell}}$ be a monomial of $A^s(k, \ell)$. Setting $\tilde u=x_{i_1}x_{i_2}\cdots x_{i_{\ell-1}}$,
then $\vert[x_1x_2\cdots x_{\ell-1}x_{k+\ell},u]\vert$ is a sum of $t$ suitable binomial coefficients, where \\
\[t= \begin{cases}
i_1, & \textit{if} \,\,\,\ \Gap(\tilde u)=\emptyset,\\ \\

i_1+\sum_{s=1}^{p} \wdt(g_s\gap(\tilde u)), &\textit{if} \,\,\,\ \Gap(\tilde u)=\{g_1,\ldots,g_p\}\neq \emptyset.
\end{cases}
\]
\end{thm}
\begin{proof} Set $m=\vert[x_1x_2\cdots x_{\ell-1}x_{k+\ell},u]\vert$. $m$ is the number of all monomials $w\in A^s(k, \ell)$ such that $w\ge_{slex}u$. By Lemma \ref{lem:count}, the binomial coefficient $\binom{k+\ell -1}{\ell -1}=\vert A^s(k, \ell)\vert$ can be decomposed as a sum of $k+1$ binomial coefficients, as follows: 
\begin{equation}\label{eq:binom1}
\resizebox{0.9\hsize}{!}{$\displaystyle \binom{k+\ell -1}{\ell -1}=\sum_{j=1}^{k+1}{\binom{k+\ell-1-j}{\ell-2}}=\binom{k+\ell -2}{\ell -2}+\binom{k+\ell -3}{\ell -2}+\cdots+\binom{\ell -2}{\ell -2}.$}
\end{equation}
One can observe that $\binom{k+\ell -2}{\ell -2}$  counts the monomials $w\in A^s(k, \ell)$ such that $\min(w)$ $=$ $1$, the binomial coefficient $\binom{k+\ell -3}{\ell -2}$ counts the monomials $w\in A^s(k, \ell)$ such that $\min(w)=2$. In general, the binomial coefficient $\binom{k+\ell -i}{\ell -2}$ counts the monomials $w\in A^s(k, \ell)$ such that $\min(w)=i-1$, for $i=4,\ldots,k+2$. Note that $\binom{\ell -2}{\ell -2} = \binom{k+\ell -(k+2)}{\ell -2}$ counts the monomials $w\in A^s(k, \ell)$ with $\min(w)=k+1$. Indeed, there exists only a monomial $w$ of such a type. It is $w=x_{k+1}x_{k+2}\cdots x_{k+\ell} = \min A^s(k, \ell)$.
It is clear that all monomials $w\in A^s(k, \ell)$ with $\min(w)< i_1=\min(\tilde u)= \min(u)$ are greater than $u$. Hence, the first $i_1-1$ binomial coefficients in (\ref{eq:binom1}) give a contribute for the computation of $m$.\\
We need to distinguish two cases:  $\Gap(\tilde u) = \emptyset$, $\Gap(\tilde u) \neq \emptyset$.\\
Note that $\Gap(\tilde u)=\Gap(u)$, or $\Gap(\tilde u)=\Gap(u)-1$.\\
{\bf Case 1.} 
Let $\Gap(\tilde u)=\emptyset$. In such a case, $u$ is the greatest monomial of $A^s(k, \ell)$
with $\min(u) =i_1$. More precisely, the following sum of binomial coefficients
\begin{equation}\label{eq:first}
\sum_{j=1}^{i_1-1}{\binom{k+\ell-1-j}{\ell-2}}
\end{equation}
gives the number of all monomials $w\in A^s(k, \ell)$ greater than $u$. Since $i_1, i_2, \ldots, i_{\ell}$ are consecutive integers, then other monomials greater than $u$ which are different from the $w$'s counted by (\ref{eq:first}) do not exist. Hence,
\[m= \vert[x_1x_2\cdots x_{\ell-1}x_{k+\ell},u]\vert = \sum_{j=1}^{i_1-1}{\binom{k+\ell-1-j}{\ell-2}} +1.
\]
On the other hand, $1 = \binom 00$, and consequently $m$ is the sum of $t=i_1-1+1=i_1= \min(\tilde u)=\min(u)$ binomial coefficients.\\
{\bf Case 2.} Let $\Gap(\tilde u)=\{g_1,\ldots,g_p\}$, $p\geq 1$. 
It is worthy to point out that the existence of the gaps $g_j$ ($j=1, \ldots, p$) implies that $i_{g_j+1}-i_{g_j}-1>0$, \emph{i.e.}, $\supp(\tilde u)\cap \{q: i_{g_j} < q < i_{g_j+1}\}= \emptyset$, for all $j\in [p]$. Thus, all monomials $w \in A^s(k, \ell)$ of the type $x_{i_1}x_{i_2}\cdots x_{i_{g_j}}z$, where $z$ is a monomial of degree $\ell-g_j$ and $\m(z)=k+\ell$ such that 
$\supp (z) \cap \{q: i_{g_j} < q < i_{g_j+1}\}\neq \emptyset$, are greater than $u$.\\
It is clear that all these monomials make up the left segment $[x_1x_2\cdots x_{\ell-1}x_{k+\ell},u)$.\\
Let us consider the $i_1$--th binomial in (\ref{eq:binom1}):
\begin{equation}\label{eq:binom2}
\binom{k+\ell -1-i_1}{\ell -2}=\sum_{j=1}^{k+1}{\binom{k+\ell-1-i_1-j}{\ell-3}}.
\end{equation}
In order to compute all monomials $w$ of the type $x_{i_1}x_{i_2}\cdots x_{i_{g_1}}z$, we need to evaluate $g_1$ successive binomial decompositions until the next one:   
\begin{align}\label{eq:binomg1}
\binom{k+\ell-i_{g_1}-1}{\ell+i_1-i_{g_1}-2} = \sum_{j=1}^{k-i_1+1}{\binom{k+\ell-i_{g_1}-1-j}{\ell+i_1-i_{g_1}-3}}.
\end{align}
The sum of the first $\wdt(g_1\gap(\tilde u))= i_{g_1+1}-i_{g_1}-1$ binomial coefficients in (\ref{eq:binomg1}) gives the number of all monomials $w \in A^s(k, \ell)$ we are looking for.\\
In order to compute all monomials $w\in A^s(k, \ell)$ of the type $x_{i_1}x_{i_2}\cdots x_{i_{g_2}}z$, we consider the $(\wdt(g_1\gap(\tilde u))-1)$--th binomial in (\ref{eq:binomg1}):\\

\resizebox{0.96\hsize}{!}{
\begin{math}
\begin{aligned}
\binom{k+\ell-i_{g_1}-1-\wdt(g_1\gap(\tilde u))-1}{\ell+i_1-i_{g_1}-3} & = \binom{k+\ell-i_{g_1+1}-1}{\ell+i_1-i_{g_1}-3}=\\
 & =\sum_{j=1}^{k-i_1+i_{g_1}- i_{g_1+1}+2} \binom{k+\ell-i_{g_1+1}-1-j}{\ell+i_1-i_{g_1}-4}.
\end{aligned}
\end{math}
}

Hence, evaluating the $i_{g_2}-i_{g_1+1}$ successive binomial decompositions until  
\begin{equation}\label{eq:binomg2}
\resizebox{0.9\hsize}{!}{$\displaystyle \binom{k+\ell-i_{g_2}-1}{\ell+i_1-i_{g_1}-i_{g_2}+i_{g_1+1}-3} = 
\sum_{j=1}^{k-i_1+i_{g_1}-i_{g_1+1}+2}{\binom{k+\ell-i_{g_2}-1-j}{\ell+i_1-i_{g_1}-i_{g_2}+i_{g_1+1}-4}},$}
\end{equation}
the number of all required monomials $w \in A^s(k, \ell)$ will be given by the sum of the first $\wdt(g_2\gap(\tilde u))= i_{g_2+1}-i_{g_2}-1$ binomial coefficients in (\ref{eq:binomg2}).\\
The procedure can be iterated for all $g_j\in\Gap(\tilde u)$, $j\ge 3$.

Finally, $\vert[x_1x_2\cdots x_{\ell-1}x_{k+\ell},u)\vert=i_1-1+\sum_{s=1}^{p}\wdt(g_s\gap(\tilde u))$.
Hence, in order to get $\vert[x_1x_2\cdots x_{\ell-1}x_{k+\ell},u]\vert$,  we must take into account the binomial $\binom{0}{0}$ which counts the monomial $u$:
\[t= i_1-1+\sum_{s=1}^{p}\wdt(g_s\gap(\tilde u))+1=i_1+\sum_{s=1}^{p}\wdt(g_s\gap(\tilde u)).\]
The assertion follows.
\end{proof}
\begin{rem} Our choice to focus on the monomial $\tilde u=x_{i_1}x_{i_2}\cdots x_{i_{\ell-1}}$, instead of $u$,  in Theorem \ref{thm:count} is due to the fact that if 
$i_{\ell-1} < k+\ell -1$, \emph{i.e.}, $\Gap(\tilde u)=\Gap(u)-1$, then all monomials $z \in A^s(k, \ell)$ such that $k+\ell -1\in \supp (z)$ are smaller than $u$, with respect to $\ge_{slex}$.
\end{rem}

Next example illustrates Theorem \ref{thm:count}.
\begin{ex} \label{Ex:count1}
Let $S= K[x_1, \ldots, x_9]$ and consider the monomial $u=x_2x_5x_7x_8$. Set $\tilde u= x_2x_5x_7$. From Remark \ref{rem:equ}, $\vert A^s(4, 4)\vert = \binom{7}{3}=35$. 
In order to compute $m=\vert [x_1x_2x_3x_8, u]\vert$, we consider the following binomial decomposition:
\[\binom{7}{3}=\binom{6}{2}+\binom{5}{2}+\binom{4}{2}+\binom{3}{2}+\binom{2}{2}.\]
Since, $\min(u) = 2$, then all monomials $w\in A^s(4, 4)$ with $\min(w)=1$ are greater than $u$, so we must take into account the binomial coefficient \fbox{$\binom{6}{2}=15$} for the computation of $m$.

Now, let us consider the following binomial decomposition:
\[\binom{5}{2}=\binom{4}{1}+\binom{3}{1}+\binom{2}{1}+\binom{1}{1}.\]
Since $\Gap(\tilde u)=\{1,2\}$ and $\wdt (1\gap(\tilde u))=2$, the sum \fbox{$\binom{4}{1}+\binom{3}{1}=7$} gives the number of all monomials of the type $x_2z\in A^s(4,4)$, with $z$ squarefree monomial of degree $3$ and $\m(z)=8$ such that $\supp (z) \cap \{q: 2 < q <5\}\neq \emptyset$. 

At this stage, we have \fbox{$15+7=22$} monomials. 

The next decomposition we need to consider is
\[
\binom{2}{1}=\binom{1}{0}+\binom{0}{0}.
\] 

Since $2 \in \Gap(\tilde u)$, and $\wdt(2\gap(\tilde u))=1$, we must take into account \fbox{$\binom{1}{0}=1$}. 

Finally, we have obtained \fbox{$22+1=23$} monomials of $A^s(4, 4)$ greater than $u$, and so  
$m=\vert [x_1x_2x_3x_8, u]\vert = 23+1=24$.

The following scheme summarizes the  previous calculations.
\begin{align*}
\tbinom{7}{3}=\boxed{\mathbf{\tbinom{6}{2}}} + & \tbinom{5}{2}+ \tbinom{4}{2}+ \tbinom{3}{2}+ \tbinom{2}{2} \\
 & \tbinom{5}{2}=\boxed{\mathbf{\tbinom{4}{1}}+\mathbf{\tbinom{3}{1}}}+ 
 \begin{aligned}[t]& \tbinom{2}{1} + \tbinom{1}{1}\\
 & \tbinom{2}{1} = \boxed{\mathbf{\tbinom{1}{0}}}+\tbinom{0}{0}.
 \end{aligned}
\end{align*}

Now, consider the monomial $v=x_3x_4x_7x_8$ and let $\tilde v = x_3x_4x_7$.
Proceeding as before, since $\Gap(\tilde v)= \{1\}$, then $\vert [x_1x_2x_3x_8, u]\vert = 27+1$, where $27$ is given by the sum of the highlighted binomial coefficients in the next scheme:

\begin{align*}
\tbinom{7}{3}=\boxed{\mathbf{\tbinom{6}{2}}+\mathbf{\tbinom{5}{2}}}+ & \tbinom{4}{2}+\tbinom{3}{2}+\tbinom{2}{2}\\
  & \tbinom{4}{2}=
  \begin{aligned}[t]& \tbinom{3}{1}+\tbinom{2}{1}+\tbinom{1}{1}\\
       & \tbinom{3}{1}=\boxed{\mathbf{\tbinom{2}{0}}+\mathbf{\tbinom{1}{0}}}+\tbinom{0}{0}.
  \end{aligned}
\end{align*}

Here is the list of all monomials which come into play for $u$ and $v$:
\begin{flalign*}
x_1x_2x_3x_8, x_1x_2x_4x_8, x_1x_2x_5x_8, x_1x_2x_6x_8, x_1x_2x_7x_8,\\
x_1x_3x_4x_8, x_1x_3x_5x_8, x_1x_3x_6x_8, x_1x_3x_7x_8,\\
x_1x_4x_5x_8, x_1x_4x_6x_8, x_1x_4x_7x_8,\\
x_1x_5x_6x_8, x_1x_5x_7x_8,\\
x_1x_6x_7x_8,\\
x_2x_3x_4x_8, x_2x_3x_5x_8, x_2x_3x_6x_8, x_2x_3x_7x_8,\\
x_2x_4x_5x_8, x_2x_4x_6x_8, x_2x_4x_7x_8,\\
x_2x_5x_6x_8, \mathbf{x_2x_5x_7x_8},\\
x_2x_6x_7x_8,\\
x_3x_4x_5x_8, x_3x_4x_6x_8, \mathbf{x_3x_4x_7x_8},\\
x_3x_5x_6x_8, x_3x_5x_7x_8,\\
x_3x_6x_7x_8,\\
x_4x_5x_6x_8, x_4x_5x_7x_8,\\
x_4x_6x_7x_8,\\
x_5x_6x_7x_8
\end{flalign*}

\end{ex}

Now, let $u_1, \ldots u_r$ be squarefree monomials of degree $q$ of $S$. We denote by $B(u_1, \ldots,u_r)$ the smallest squarefree strongly stable set of $\Mon_q^s(S)$ containing the monomials $u_1,\ldots, u_r$.

It is well known that if $q<n$, $\Shad(B(u_1, \ldots, u_r))$ is a squarefree strongly stable set of monomials of degree $q+1$ of $S$, and consequently $\Shad^i(B(u_1, \ldots, u_r))$ is a squarefree strongly stable set of degree $q+i$, for $1\leq i\leq n-q$.\\

Now, let $(k_1, \ell_1)$ and $(k_2, \ell_2)$ be two pairs of positive integers such that $k_1> k_2$, $\ell_1 < \ell_2$, $k_i+\ell_i\le n$ ($i=1, 2$). If $u_1, \ldots, u_r \in \Mon_{\ell_1}^s(S)$ are squarefree monomials of $S$ such that $\m(u_j) = k_1+\ell_1$, $j=1, \ldots, r$, we define the following set:
\[\BShad(u_1, \ldots, u_r)_{(k_2, \ell_2)} = \{v \in \Shad^{\ell_2-\ell_1}(B(u_1, \ldots, u_r)) : \m(v) \le k_2+\ell_2\}.\]

One can quickly observe that $\BShad(u_1, \ldots, u_r)_{(k_2, \ell_2)}$ is a squarefree strongly stable set of degree $\ell_2$ of $S$.

\begin{rem}  It is worthy to underline that if one wants to compute the minimum of $\BShad(u_1, \ldots, u_r)_{(k_2, \ell_2)}$, it is sufficient to determine $\min\BShad(u_r)_{(k_2, \ell_2)}$. Furthermore, in order to obtain such a monomial, one can suitably manage the integers in $\supp(u_r)$, as we will see in a while.
\end{rem}

\begin{defn} \rm Let $u$ be a squarefree monomial of degree $q$ of $S$, $q<n$. Let $p\le n$ a positive integer such that $[p]\setminus \supp(u)\neq \emptyset$ and $\{j_1, \ldots, j_t\}$ a subset of $[p]\setminus \supp(u)$, with $j_1 < j_2 < \cdots < j_t$, $q+t \le n$.
The monomial $x_{j_1} \cdots x_{j_t} u \in \Mon_{q+t}^s(S)$ is called the joint of $u$ with the variables $x_{j_1}, \ldots, x_{j_t}$. 
\end{defn}

\begin{ex} Let $u = x_1x_3x_6x_8\in K[x_1, \ldots, x_9]$. Let $p=7$ and consider the set $\{2,4,7\} \subset [7]\setminus \{1, 3, 6, 8\}$. The joint of $u$ with $x_2, x_4, x_7$ is the squarefree monomial $x_1x_2x_3x_4x_6x_7x_8\in \Mon^s_7(S)$.
\end{ex}

With the same notations as before, we give the construction of the monomial $\min\BShad(u)_{(k_2, \ell_2)}$ for a given squarefree monomial $u\in A^s(k_1, \ell_1)$.
\begin{Const}  \label{const} Let $(k_1, \ell_1)$ and $(k_2, \ell_2)$ be two pairs of positive integers such that $k_1> k_2$, $2\leq \ell_1 < \ell_2$ and $k_i+\ell_i\le n$, for $i=1,2$. Let $u = x_{i_1}\cdots x_{i_{\ell_1}}$ be a squarefree monomial of $A^s(k_1, \ell_1)$.
Assume $i_t$ to be the greatest integer belonging to $\supp(u)$ such that $i_t<k_2+\ell_2$, and write
\[u=x_{i_1}\cdots x_{i_t} \cdots x_{i_{\ell_1}}.\]
Let us consider the monomial $\overline u = x_{i_1}\cdots x_{i_t}$ and let $j_1, \ldots, j_{\ell_2-t}$ be the greatest integers belonging to $[k_2+\ell_2]\setminus \supp(\overline u)$. Then,  
\[\min\BShad(u)_{(k_2, \ell_2)}= x_{j_1} \cdots x_{j_{\ell_2-t}}\overline u \in A^s(k_2, \ell_2).\]
\end{Const}
Construction \ref{const} assures the correctness of the next algorithm.

\begin{center}
\scalebox{0.9}{
\begin{algorithm}[H]
  \caption{Computation of $\min\BShad(u)_{(k, \ell)}$}
  \label{alg:1}
  \KwIn{Polynomial ring $S$, monomial $u$, positive integer $k$, positive integer $\ell$}
  \KwOut{monomial $v$}
  \SetKw{Error}{error}
  \Begin{
  $j \gets k+\ell$\;  
  $t \gets \mid\{i\in \supp(u)\ :\ i<j\}\mid$\;
  $v \gets $ the first $t$ variables of $u$\;
  
  $q \gets \ell-t$\;
  \While{$q>0$}{    
    \If{$j\notin \supp(v)$}{
      \eIf{$j>0$}{  
        $v \gets v*S_j$\;
      }{
        \Error{no monomial}\;
      } 
      $q \gets q-1$\;               
    }      
    $j \gets j-1$\;
  }   
  \Return $v$\;
}
\end{algorithm}
}
\end{center}
\vspace{0,5cm}

\begin{lem}\label{lem:next}
Take two pairs of positive integers $(k_1, \ell_1)$ and $(k_2, \ell_2)$  such that $k_1> k_2$, $2\le \ell_1 < \ell_2$ with $k_i+\ell_i\le n$, for $i=1,2$. Let $u$ be a squarefree monomial of degree $\ell_1$ with $\m(u) = k_1+\ell_1$ and let $v= \min\BShad(u)_{(k_2, \ell_2)}$.
If $\Gap(v)\neq \emptyset$, then there exists a monomial 
$w \in A^s(k_2, \ell_2)\setminus \BShad(u)_{(k_2, \ell_2)}$ .
\end{lem}
\begin{proof} Let 
\[v=\min\BShad(u)_{(k_2, \ell_2)}= x_{r_1}\cdots x_{r_{\ell_2}}.\]
One has $\m(v)= k_2+\ell_2$. 
Assume $p =  \max \Gap(v)$,  then the greatest squarefree monomial following $v$ in the squarefree lex order is  
$x_{r_1} \cdots x_{r_{p-1}}x_{r_p+1}\cdots$ $x_{r_p+ \ell_2-p+1}$, with $r_p+ \ell_2-p+1\le k_2+\ell_2$. Hence, if $r_p+ \ell_2-p+1= k_2+\ell_2$, we choose
\[w=  x_{r_1} \cdots x_{r_{p-1}}x_{r_p+1}\cdots x_{r_p+ \ell_2-p+1}.\]
Otherwise, if $r_p+ \ell_2-p+1< k_2+\ell_2$, let
\[w=x_{r_1} \cdots x_{r_{p-1}}x_{r_p+1}\cdots x_{r_p+ \ell_2-p}x_{k_2+\ell_2}.\]
\end{proof}

Next pseudocode describes the procedure in Lemma \ref{lem:next}.

\begin{center}
\scalebox{0.9}{
\begin{algorithm}[H]
  \caption{Computation of the next monomial smaller than a given $u$ in $A^s(k,\ell)$}
  \label{alg:2}
  \KwIn{Polynomial ring $S$, monomial $u$}
  \KwOut{monomial $w$}
  \SetKw{Error}{error}
  \Begin{
  $m \gets \max\ \supp(u)$\;
  $\ell \gets \deg(u)$\;   
  \eIf{$\Gap(u)\neq\emptyset$}{
    $t \gets \max\ \Gap(u)$\;
    $w \gets $ the first $t-1$ variables of $u$\;
    $j \gets $ index of variable of $u$ at position $t$\;
  
    \ForEach{$i\in \{1\, .\, .\, \ell-t\}$}{
      $j \gets j+1$ \;
      $w \gets w*S_j$\;
    }  
    $w \gets w*S_{m}$\; 
  }{
    \Error{no monomial}\;
  } 
  \Return $w$\;
}
\end{algorithm}
}
\end{center}

The discussion below is significant for solving Problem \ref{probl:1}.

\begin{Discus} \label{disc} \em Let $(k_1, \ell_1)$ and $(k_2, \ell_2)$ be two pairs of positive integers such that $k_1> k_2$, $2\leq \ell_1 < \ell_2$ with $k_i+\ell_i\le n$ ($i=1,2$) and let $a_1,a_2$ be two positive integers. 

Let $T$ be a segment of $A^s(k_2, \ell_2)$ of cardinality $a_2< \binom{k_2+\ell_2 -1}{\ell_2 -1}$. We want to determine the admissible values for $a_1\leq \binom{k_1+\ell_1 -1}{\ell_1 -1}$ so that there exists a segment $[u_1, u_{a_1}]$ of $A^s(k_1, \ell_1)$ of cardinality $a_1$ and such that $\BShad([u_1, u_{a_1}])_{(k_2, \ell_2)} $
$\nsupseteq T$.
It is clear that it should be $a_1< \binom{k_1+\ell_1 -1}{\ell_1 -1}$.

Now, set $T=[z_1, z_{a_2}]$, and assume $T\nsubseteq \BShad([u_1, u_{a_1}])_{(k_2, \ell_2)}$. Let $v_1\in A^s(k_1, \ell_1)$ be the smallest monomial such that $z_1 \notin \BShad(v_1)_{(k_2, \ell_2)}$.
Such a monomial allows us to determine the bound on $a_1$ for which there exists the segment $T$.

Indeed, we can compute the following cardinalities (Theorem \ref{thm:count}):
\begin{align*}
n_1 &= \vert \{u\in A^s(k_1, \ell_1): u\ge v_1\}\vert = \vert[x_1x_2 \cdots x_{\ell_1-1}x_{k_1+\ell_1},v_1]\vert,\\
p_1&= \vert \{v\in A^s(k_1, \ell_1): v> u_1\}\vert =\vert[x_1x_2 \cdots x_{\ell_1-1}x_{k_1+\ell_1},u_1)\vert .
\end{align*}

Hence, since $[u_1, u_{a_1}]\subseteq [x_1x_2 \cdots x_{\ell_1-1}x_{k_1+\ell_1},v_1]$, we get the following coarse bound for $a_1$:
\[a_1 \le n_1;\]
then, we can refine such a bound via $p_1$ as follows:
\[
a_1\le n_1-p_1.
\]

One can notice, that if $u_1 = \max A^s(k_1, \ell_1)$, then $p_1=0$.
\end{Discus}

\begin{ex}
Given $S= K[x_1, \ldots, x_{10}]$, let us consider the pairs of positive integers $(5,4)$ and $(2,6)$, the positive integers 
$a_1 = 8$ and $a_2 = 6$, and the following segment of $A^s(5,4)$ of cardinality $a_1 = 8$:
\begin{align*}
[x_1x_3x_4x_9, x_1x_4x_7x_9] = \{x_1x_3x_4x_9, x_1x_3x_5x_9, & x_1x_3x_6x_9, x_1x_3x_7x_9, x_1x_3x_8x_9,\\ & x_1x_4x_5x_9, x_1x_4x_6x_9, x_1x_4x_7x_9\}.
\end{align*}

We want to verify if there exists a segment of $A^s(2,6)$ of cardinality $a_2=6$ not contained in $\BShad([x_1x_3x_4x_9, x_1x_4x_7x_9])_{(2, 6)}$.\\

First, from Remark \ref{rem:equ}, we know that $a_1\leq\binom{8}{3}=56$ and $a_2\leq\binom{7}{5}=21$.\\
In order to determine \scalebox{0.97}{$p_1= \vert \{v\in A^s(5, 4): v> x_1x_3x_4x_9\}\vert=$
$\vert[x_1x_2x_3x_9,x_1x_3x_4x_9)\vert$}, we need to consider 
a suitable sequence of binomial decompositions. The first  binomial decomposition that we have to examine is 
\[\binom{8}{3}=\binom{7}{2}+\binom{6}{2}+\binom{5}{2}+\binom{4}{2}+\binom{3}{2}+\binom{2}{2}.\]
Then, applying the procedure described in Theorem \ref{thm:count} (see also Example \ref{Ex:count1}), we obtain the following sequence of binomial decompositions,
\begin{align*}
\tbinom{8}{3}= & \tbinom{7}{2}+\tbinom{6}{2}+\tbinom{5}{2}+\tbinom{4}{2}+\tbinom{3}{2}+\tbinom{2}{2}\\
  & \tbinom{7}{2}= \boxed{\mathbf{\tbinom{6}{1}}}+\tbinom{5}{1}+\tbinom{4}{1}+\tbinom{3}{1}+\tbinom{2}{1}+\tbinom{1}{1},
\end{align*}
whereupon $p_1= 6$.

In order to compute $n_1$, we consider the set $A_2$ consisting of the smallest $a_2=6$ monomials of $A^s(2,6)$: 
\begin{align*}
A_2=\{x_2x_3x_4x_5x_6x_8, x_2x_3x_4x_5x_7x_8, &x_2x_3x_4x_6x_7x_8, x_2x_3x_5x_6x_7x_8,\\
&x_2x_4x_5x_6x_7x_8, x_3x_4x_5x_6x_7x_8\}. 
\end{align*}
These monomials can be found using the ``reversal'' of Algorithm~\ref{alg:2}.

The smallest monomial $z$ of $A^s(5,4)$ such that $\max A_2 =$ $x_2x_3x_4x_5x_6x_8 \notin \BShad(z)_{(2,6)}$ is $z=x_1x_7x_8x_9$. The number of all monomials $w\in A^s(5,4)$ greater than or equal to $z$ is determined by the following binomial sequences:
\begin{align*}
\tbinom{8}{3}= & \tbinom{7}{2}+\tbinom{6}{2}+\tbinom{5}{2}+\tbinom{4}{2}+\tbinom{3}{2}+\tbinom{2}{2}\\
 & \tbinom{7}{2}= \boxed{\mathbf{\tbinom{6}{1}}+\mathbf{\tbinom{5}{1}}+\mathbf{\tbinom{4}{1}}+\mathbf{\tbinom{3}{1}}+\mathbf{\tbinom{2}{1}}}+\tbinom{1}{1}.   
\end{align*}

Hence, we have $n_1=(6+5+4+3+2)+1=21$ monomials.
Finally, we have $a_1\leq n_1-p_1=21-6=15$. 

For $a_1=15$, then a segment of $A^s(2,6)$ of length $a_2=6$ is 
\[A_2 = [x_2x_3x_4x_5x_6x_8, x_3x_4x_5x_6x_7x_8].\]
\end{ex}

Discussion \ref{disc} yields the following result.
\begin{thm} \label{thm:constru} Consider three positive integers $n\ge 5$, $\ell_1\ge 3$ and $1\le r \le n-\ell_1$, $r$ pairs of positive integers $(k_1, \ell_1)$,  $\ldots$, $(k_r, \ell_r)$ such that $n-3 \ge k_1 >k_2 > \cdots >k_r\ge 2$ and $2\le \ell_1 < \ell_2 < \cdots < \ell_r$, $k_i+\ell_i\le n$ ($i=1, \ldots, r$), and $r$ positive integers $a_1, \ldots, a_r$. Let $K$ be a field of characteristic zero. The following conditions are equivalent:
\begin{enumerate}
\item[\rm(1)] There exists a squarefree graded ideal $J$ of $S=K[x_1,\dots,x_n]$ with\\ $\beta_{k_1, k_1+\ell_1}(J) = a_1$, $\ldots$, $\beta_{k_r, k_r+\ell_r}(J) = a_r$ as extremal Betti numbers.
\item[\rm(2)] There exists a squarefree strongly stable ideal $I$ of $S=K[x_1,\dots,x_n]$ with $\beta_{k_1, k_1+\ell_1}(I) = a_1$, $\ldots$, $\beta_{k_r, k_r+\ell_r}(I) = a_r$ as extremal Betti numbers.
\item[\rm(3)] Setting
\begin{enumerate}
\item[\rm(i)] $v_r = x_{k_r+1}\cdots x_{k_r+\ell_r}$,
\item[] $A_r=[w_r,v_r]$, with $w_r\in A^s(k_r, \ell_r)$ and such that $\vert A_r\vert = a_r$;
\item[\rm(ii)] for $i=1,\ldots, r-1$, 
\item[] \scalebox{0.97}{$v_{r-i} = \min\{u\in A^s(k_{r-i}, \ell_{r-i}): \max A_{r-i+1} \notin \BShad(u)_{(k_{r-i+1}, \ell_{r-i+1})}\}$}, 
\item[] $A_{r-i} = [w_{r-i},v_{r-i}]$, with $w_{r-i}\in A^s(k_{r-i}, \ell_{r-i})$ and such that $\vert A_{r-i} \vert = a_{r-i}$;
\item[\rm(iii)] for $i=1, \ldots, r$, $n_i = \vert \{u\in A^s(k_i, \ell_i): u\ge v_i\}\vert$, 
then the integers $a_i$ satisfy the following conditions:
\[a_i \le n_i.\] 
If $a_i=\vert [u_{i, 1}, u_{i, a_i}]\vert$, $u_{i, j}\in A^s(k_i, \ell_i)$ ($j=1, \ldots, a_i$) and $p_i= \vert \{v\in A^s(k_i, \ell_i): v> u_{i, 1}\}\vert$, then  $a_i \le n_i-p_i$, for $i=1, \ldots, r$.
\end{enumerate}
\end{enumerate}
\end{thm}
\begin{proof} (1) $\Leftrightarrow$ (2). See \cite{AHH3} and the introduction in this paper.\\
(2) $\Rightarrow$ (3). It follows applying iteratively Discussion \ref{disc}, for $i=1, \ldots, r$. Note that $v_r=\min A^s(k_r,\ell_r)$, and consequently $n_r = \binom{k_r+\ell_r -1}{\ell_r -1}$; whereas $p_1=0$.\\
(3) $\Rightarrow$ (2). We construct a squarefree strongly stable ideal $I$ of $S$ generated in degrees $\ell_1, \ldots, \ell_r$ as follows:
\begin{enumerate}
\item[-] $G(I)_{\ell_1} =B(u_{1, 1}, \ldots, u_{1, a_1})$;
\item[-] $G(I)_{\ell_2} = B(u_{2, 1}, \ldots, u_{2, a_2}) \setminus \BShad^{\ell_2-\ell_1}(G(I)_{\ell_1})_{(k_2, \ell_2)}$;
\item[-] $G(I)_{\ell_i} = B(u_{i, 1}, \ldots, u_{i, a_i}) \setminus \BShad^{\ell_i-\ell_{i-1}}(\Mon^s(I_{\ell_{i-1}}))_{(k_i, \ell_i)}$, for \scalebox{0.9}{$i=3, \ldots, r$}, where $\Mon^s(I_{\ell_{i-1}})$ is the set of all squarefree monomials of degree $\ell_{i-1}$ belonging to $I_{\ell_{i-1}}$.
\end{enumerate}
The monomials $u_{i,1},\ldots,u_{i,a_i}$, for $i=1,\ldots,r$, are the \emph{basic monomials} of $I$.
\end{proof}

\begin{rem} A similar statement can be formulated in the case $\ell_1 =2$ and $n\ge 5$.
\end{rem}

%

Next example illustrates Theorem \ref{thm:constru}.

\begin{ex}
Let $n=11$, $r=4$, $\mathcal C=\{(8,3),(4,5),(3,6),(2,9)\}$ and $a$ $=$ $(a_1, a_2, a_3, a_4)=(7,5,2,2)$. We want to construct a squarefree strongly stable ideal $I$ of $S= K[x_1, \ldots, x_{11}]$ generated in degrees 3,5,6,9 and such that $\C(I)=\mathcal C$, $a(I)=a$.\\

With the same notations as in Theorem \ref{thm:constru}, before starting the construction of the ideal, we verify if the coarse bounds are satisfied for each $a_i$, $i=1,\ldots,4$.

First of all, $v_4 = x_3x_4x_5x_6x_7x_8x_9x_{10}x_{11}$ and $n_4 = \vert[x_1x_2x_3x_4x_5x_6x_7x_8x_{11}, v_4]\vert$ $=$ $\binom{10}{8}=45$. 
Hence, $a_4=2\leq n_4$.

Moreover, $A_4=\{x_2x_4x_5x_6x_7x_8x_9x_{10}x_{11}, x_3x_4x_5x_6x_7x_8x_9x_{10}x_{11}\}$,\\
$v_3=$ $x_2x_3x_6x_7x_8x_9$, and from the binomial decompositions
\begin{align*}
\tbinom{8}{5}= \boxed{\mathbf{\tbinom{7}{4}}}+ & \tbinom{6}{4}+\tbinom{5}{4}+\tbinom{4}{4}\\
 & \tbinom{6}{4}= 
 \begin{aligned}[t] & \tbinom{5}{3}+\tbinom{4}{3}+\tbinom{3}{3}\\
 &\tbinom{5}{3}= \boxed{\mathbf{\tbinom{4}{2}}+\mathbf{\tbinom{3}{2}}}+\tbinom{2}{2}
 \end{aligned}
\end{align*}     
we obtain $a_3=2\leq n_3=\vert[x_1x_2x_3x_4x_5x_9, v_3]\vert =35+(6+3)+1=45$.

Furthermore, $A_3=\{x_2x_3x_5x_7x_8x_9, x_2x_3x_6x_7x_8x_9\}$ and $v_2=x_2x_3x_5x_6x_9$. From the binomial decompositions
\begin{align*}
\tbinom{8}{4}= \boxed{\mathbf{\tbinom{7}{3}}}+ & \tbinom{6}{3}+\tbinom{5}{3}+\tbinom{4}{3}+\tbinom{3}{3}\\
 & \tbinom{6}{3}= 
 \begin{aligned}[t] & \tbinom{5}{2}+\tbinom{4}{2}+\tbinom{3}{2}+\tbinom{2}{2}\\
 & \tbinom{5}{2}= \boxed{\mathbf{\tbinom{4}{1}}}+\tbinom{3}{1}+\tbinom{2}{1}+\tbinom{1}{1}
 \end{aligned}
\end{align*}         
one has $a_2=5\leq n_2 = \vert[x_1x_2x_3x_4x_9, v_2]\vert =(35+4)+1=40$.

Finally,  $A_2=[x_2x_3x_4x_5x_9, x_2x_3x_5x_6x_9]$ $=\{x_2x_3x_4x_5x_9,$ $x_2x_3x_4x_6x_9,$\\ $ x_2x_3x_4x_7x_9, $ $x_2x_3x_4x_8x_9,$ $x_2x_3x_5x_6x_9 \}$ 
and $v_1=x_1x_{10}x_{11}.$
The binomial decompositions
\begin{align*}
\tbinom{10}{2}= & \tbinom{9}{1}+ \tbinom{8}{1}+\tbinom{7}{1}+\tbinom{6}{1}+\tbinom{5}{1}+\tbinom{4}{1}+\tbinom{3}{1}+\tbinom{2}{1}+\tbinom{1}{1}\\
& \tbinom{9}{1}= \boxed{\mathbf{\tbinom{8}{0}}+\mathbf{\tbinom{7}{0}}+\mathbf{\tbinom{6}{0}}+\mathbf{\tbinom{5}{0}}+\mathbf{\tbinom{4}{0}}+\mathbf{\tbinom{3}{0}}+\mathbf{\tbinom{2}{0}}+\mathbf{\tbinom{1}{1}}}+\tbinom{0}{0}
\end{align*}

imply $a_1=7\leq n_1= \vert [x_1x_2x_{11}, v_1]\vert =8+1=9$.

Now, we proceed with the construction of the ideal $I$ we are looking for, and so doing we refine the previous bounds for the $a_i$'s.

\begin{itemize}
\item[-] The greatest monomial of $A^s(8,3)$ is $x_1x_2x_{11}$. Since $p_1$ must be equal to $0$ and $a_1=7\leq n_1-p_1=9$, one can consider the greatest $a_1=7$ monomials of $A^s(8, 3)$. Such monomials can be obtained by  Algorithm~\ref{alg:2}. Hence, we set
\[
G(I)_3=B(x_1x_2x_{11}, x_1x_3x_{11}, x_1x_4x_{11}, x_1x_5x_{11}, x_1x_6x_{11}, x_1x_7x_{11}, x_1x_8x_{11}).
\]

\item[-] Let us consider the corner $(4,5)$.
By Algorithm~\ref{alg:1}, we compute the smallest monomial of $\BShad^2(G(I)_3)_{(4,5)}$, \emph{i.e.}, the monomial $x_1x_6x_7x_8x_9$; whereas, by Algorithm~\ref{alg:2}, we determine the greatest monomial of $A^s(4,5)\setminus$ $\BShad^2$ $(G(I)_3)_{(4,5)}$, \emph{i.e.}, $x_2x_3x_4x_5x_9$. Finally, from the binomial decomposition
\begin{align*}
\tbinom{8}{4}= \boxed{\mathbf{\tbinom{7}{3}}}+ \tbinom{6}{3}+\tbinom{5}{3}+\tbinom{4}{3}+\tbinom{3}{3}
\end{align*}
it follows that $p_2=\vert [x_1x_2x_3x_4x_9,x_2x_3x_4x_5x_9)\vert=35$. Hence, $n_2-p_2=40-35=5$ monomials are available. 
Therefore, since $a_2=5$, we set
\[
G(I)_5=B(x_2x_3x_4x_5x_9, x_2x_3x_4x_6x_9, x_2x_3x_4x_7x_9, x_2x_3x_4x_8x_9, x_2x_3x_5x_6x_9).
\]    

\item[-] Let us consider the corner $(3,6)$. One has $\min \BShad(G(I)_5)_{(3,6)}=$\\ $x_2x_3x_5x_6x_8x_9$ and $\max (A^s(3,6)\setminus \BShad(G(I)_5)_{(3,6)}$ $=x_2x_3x_5x_7x_8x_9$, and from 
\begin{align*}
\tbinom{8}{5}= \boxed{\mathbf{\tbinom{7}{4}}} + & \tbinom{6}{4} + \tbinom{5}{4} +\tbinom{4}{4}\\
 & \tbinom{6}{4}= 
 \begin{aligned}[t] & \tbinom{5}{3} + \tbinom{4}{3} +\tbinom{3}{3} \\
 & \tbinom{5}{3}= \boxed{\mathbf{\tbinom{4}{2}}} + 
     \begin{aligned}[t] & \tbinom{3}{2} +\tbinom{2}{2} \\
           & \tbinom{3}{2}= \boxed{\mathbf{\tbinom{2}{1}}} +\tbinom{1}{1}
      \end{aligned}
  \end{aligned}
\end{align*}
                                
we have
$p_3=\vert [x_1x_2x_3x_4x_5x_9,x_2x_3x_5x_7x_8x_9)\vert=43$. Hence, $n_3-p_3=45-43=2$.\\
Since $a_3=2$, we set
\[
G(I)_6=B(x_2x_3x_5x_7x_8x_9, x_2x_3x_6x_7x_8x_9).
\]      

\item[-] If one considers the corner $(2,9)$, since 
\[\min \BShad^3(G(I)_6)_{(2,9)}=x_2x_3x_5x_6x_7x_8x_9x_{10}x_{11}\]
\[\max(A^s(2,9)\setminus \BShad^3(G(I)_6))_{(2,9)}=x_2x_4x_5x_6x_7x_8x_9x_{10}x_{11},\]
from
\begin{align*}
\tbinom{10}{8}= \boxed{\mathbf{\tbinom{9}{7}}} + & \tbinom{8}{7} + \tbinom{7}{7}\\
 & \tbinom{8}{7}= \boxed{\mathbf{\tbinom{7}{6}}} + \tbinom{6}{6}                                                            
\end{align*}

it follows $p_4=\vert [x_1x_2x_3x_4x_5x_6x_7x_8x_{11},x_2x_4x_5x_6x_7x_8x_9x_{10}x_{11})\vert=43$.\\
So $n_4-p_4=\binom{10}{8}-p_4=45-43=2$. 
Hence, since $a_4=2$, we can set
\[
G(I)_9=B(x_2x_4x_5x_6x_7x_8x_9x_{10}x_{11}, x_3x_4x_5x_6x_7x_8x_9x_{10}x_{11}).
\] 
\end{itemize}

The Betti table of the squarefree strongly stable $I$ just constructed is the following one:
\[
\begin{array}{ccccccccccc}
    &   & 0 & 1 & 2 & 3 & 4 & 5 & 6 & 7 & 8 \\
\hline
  3 & : & 42 & 217 & 553 & 861 & 875 & 587 & 252 & 63 & 7 \\
  4 & : & - & - & - & - & - & - & - & - & - \\
  5 & : & 13 & 39 & 45 & 24 & 5 & - & - & - & - \\
  6 & : & 2 & 6 & 6 & 2 & - & - & - & - & - \\
  7 & : & - & - & - & - & - & - & - & - & - \\
  8 & : & - & - & - & - & - & - & - & - & - \\
  9 & : & 2 & 4 & 2 & - & - & - & - & - & - \\
\end{array}
\]
\end{ex}
\vspace{0,4cm}

One can observe that Theorem \ref{thm:constru} assures the correctness of the next Algorithm~\ref{alg:3}.

\begin{center}
\begin{figure}
\scalebox{0.84}{
\begin{algorithm}[H]
  \caption{Computation of the basic monomials for the given data}
  \label{alg:3}
  \KwIn{Polynomial ring $S$, list of corners $\{(k_i,\ell_i)\}$, list of values $(a_i)$}
  \KwOut{list of monomials $mons$}
  \SetKw{Error}{error}
  \Begin{
  $hyp \gets $ logical conditions required as hypotheses of the Theorem\ref{thm:constru}\;
  \If{$hyp$}{
    $m \gets k_0+\ell_0$\;
    $w \gets S_1*\ldots*S_{\ell_0-1}*S_{m}$\tcp*[r]{first corner}    
    $mons \gets \{w\}$\;
    \ForEach{$j\in \{2\, .\, .\, a_0\}$}{
      $w \gets $next monomial of $w$\tcp*[r]{calling Algorithm~\ref{alg:2}}
      \eIf{no monomial}{
        \Error{no ideal}\;
      }{
        $mons \gets$ $mons\ \cup$ $\{w\}$\;
      } 
    } 
    $r \gets$ number of corners\tcp*[r]{successive corners}
    \ForEach{$i\in \{2\, .\, .\, r\}$}{
      $w \gets \min\BShad(mons)_{(k_{i-1},\ell_{i-1})}$ \tcp*[r]{calling Algorithm~\ref{alg:1}}
      \ForEach{$j\in \{1\, .\, .\, a_i\}$}{
        $w \gets$ next monomial of $w$\tcp*[r]{calling Algorithm~\ref{alg:2}}
        \eIf{no monomial}{
          \Error{no ideal}\;
        }{
          $mons \gets$ $mons\ \cup$ $\{w\}$\;        
        }
      }      
    }
  }
  \Return $mons$\;
}
\end{algorithm}
}
\end{figure}
\end{center}

We close the Section with an example that illustrates a \emph{situation} where the construction of a  squarefree strongly stable ideal is not possible.
\begin{ex}
Let $n=10$, $r=3$, $\mathcal C = \{(6,2),(5,4),(3,7)\}$ and $a=(a_1,a_2,a_3)=(2,1,4)$.
We have $\vert A^s(3,7)\vert=\binom{9}{6}=84$, so it is possible to manage $a_3=4\leq 84$ monomials.

Let us consider the set $A_2$ consisting of the smallest four monomials in $A^s(3,7)$:
\[
A_2=\{x_3x_4x_5x_7x_8x_9x_{10}, x_3x_4x_6x_7x_8x_9x_{10}, x_3x_5x_6x_7x_8x_9x_{10}, x_4x_5x_6x_7x_8x_9x_{10}\},
\]
and let us try to get the smallest monomial $z \in A^s(5,4)$ such that $x_3x_4x_5x_6x_8x_9x_{10}$ $\notin$ $\BShad(z)_{(3,7)}$. It is $z=x_2x_7x_8x_9$. Now, we compute $\vert[x_1x_2x_3x_9,z]\vert$ as bound for $a_2$:
\begin{align*}
\tbinom{8}{3}= \boxed{\mathbf{\tbinom{7}{2}}}+ & \tbinom{6}{2}+\tbinom{5}{2}+\tbinom{4}{2}+\tbinom{3}{2}+\tbinom{2}{2}\\
& \tbinom{6}{2}= \boxed{\mathbf{\tbinom{5}{1}}+\mathbf{\tbinom{4}{1}}+\mathbf{\tbinom{3}{1}}+\mathbf{\tbinom{2}{1}}}+\tbinom{1}{1}.
\end{align*}
We have $n_2=21+(5+4+3+2)+1=36$ monomials greater than $z$ and so $a_2=1\leq 36$.\\
Note that if $z$ does not exist, then it is clear that we can not go on.

Now, we try to verify the bound for $a_1$ taking into account the previous results. Consider the monomial $z\in A^s(5,4)$,
and take the greatest monomial $w$ of $A^s(6,2)$ such that $z\notin \BShad(w)_{(3,7)}$.
It is $w=x_1x_8$. We can note that $w$ is the smallest monomial of $A^s(6,2)$, \emph{i.e.}, $\vert[x_1x_8,w]\vert =1$. 

Hence, we have that $a_1\leq 1$. For this reason the requested value for $a_1=2$ is not admissible and there does not exist any squarefree monomial ideal $I$ of $K[x_1,\ldots,x_{10}]$ such that $\C(I)=\mathcal{C}$ and $a(I)=a$.

Nevertheless, there exists a squarefree monomial ideal $J$ of $S$ such that $\C(J)=\mathcal{C}$ and $a(J)=(1,1,4)$.
\end{ex}



\begin{thebibliography}{99}
\bibitem{AC} L.~Amata and M.~Crupi, {\em Computation of graded ideals with given extremal Betti numbers in a polynomial ring}, J. Symbolic Computation, {93} (2019), 120--132.

\bibitem{AHH2} A.~Aramova, J.~Herzog and T.~Hibi, {\em Squarefree lexsegment ideals},  Math. Z., {228} (1998), 353--378.

\bibitem{AHH3} A.~Aramova, J.~Herzog and T.~Hibi, {\em Shifting operations and graded Betti numbers ideals}, J. Algebraic Combin., {12} (2000), 207--222.

\bibitem{BCP} D.~Bayer, H.~Charalambous and S.~Popescu, {\em Extremal Betti numbers and Applications to Monomial Ideals}, J. Algebra, {221} (1999), 497--512.

\bibitem{BH} {W.~Bruns and J.~Herzog}, Cohen-Macaulay rings, Cambridge Studies in Advanced Mathematics, 39, Cambridge University Press, Cambridge, 1998.

\bibitem{CW} S.~M.~Cooper and S.~S.~Wagstaff, {Connection between Algebra, Combinatorics and Geometry}, Springer Proceedings in Mathematics \& Statistics, {76}, Springer--Verlag, 2014.

\bibitem{MC} M.~Crupi, {\em Extremal Betti numbers of graded modules}, J. Pure Appl. Algebra, {220} (2016), 2277--2288.

\bibitem{MC2} M.~Crupi, {\em A constructive method for standard Borel fixed submodules with given extremal Betti numbers}, Mathematics, {56}(5) (2017), 1--26.

\bibitem{MC3} M.~Crupi, {\em Computing general strongly stable modules with given extremal Betti numbers}, J. Commut. Algebra, 12(1) (2020), 53--70.

\bibitem{CF4} M.~Crupi and C.~Ferr\`{o}, {\em Squarefree monomial modules and extremal Betti numbers}, Algebra Colloq., {23(3)} (2016), 519-530.

\bibitem{CU1} M.~Crupi and R.~Utano, {\em Extremal Betti numbers of lexsegment ideals}, Lecture notes in Pure and Applied Math., Geometric and combinatorial aspects of Commutative algebra, {217} (2000), 159--164.

\bibitem{CU2} M.~Crupi and R.~Utano, {\em Extremal Betti numbers of graded ideals}, Results Math., {43} (2003), 235-244.

\bibitem{Ei} D.~Eisenbud, {Commutative Algebra with a view toward  Algebraic Geometry}, Graduate Texts in Mathematics, 150, Springer-Verlag, 1995.

\bibitem{EK} S.~Eliahou and M.~Kervaire, {\em Minimal resolutions of some monomial ideals}, J. Algebra, {129} (1990), 1--25.

\bibitem{GDS} {D.R.~Grayson and M.E.~Stillman}, Macaulay2, a software system for research in algebraic geometry, available at \url{http://www.math.uiuc.edu/Macaulay2}.

\bibitem{JT} J.~Herzog and T.~Hibi, {Monomial ideals}, Graduate texts in Mathematics, {260}, Springer--Verlag, 2011.

\bibitem{HSV} J.~Herzog, L.~Sharifan and M.~Varbaro, {\em The possible extremal Betti numbers of a homogeneous ideal}, Proc. Amer. Math. Soc., {142} (2014), 1875--1891.

\bibitem{MS}  E.~Miller and B.~Sturmfels, {Combinatorial Commutative Algebra}, Graduate Texts in Mathematics, 227, Springer-Verlag, 2005.

\end{thebibliography}

\end{document}